\documentclass[a4paper,11pt]{amsart}
\usepackage[pagewise,mathlines]{lineno}
\usepackage{graphicx}
\usepackage{amssymb}

\usepackage{amsfonts}
\usepackage[]{float}
\usepackage[colorlinks=true, linkcolor=red, citecolor=blue]{hyperref}

\usepackage[]{epsfig}
\usepackage[]{pstricks}
\usepackage{tikz}

\newtheorem{theorem}{Theorem}[section]
\newtheorem{lemma}{Lemma}[section]

\newtheorem{remark}{Remark}[section]
\newtheorem{remarks}{Remarks}[section]
\newtheorem{proposition}{Proposition}[section]

\numberwithin{equation}{section}
\newcommand{\R}{\mathbb R}

\newcommand{\Z}{{\mathbb Z}}
\newcommand{\ft}{{\mathcal{F}}}

\newcommand{\Sch}{{\mathcal{S}}}
\newcommand{\supp}{{\mbox{supp}}}
\newcommand{\sgn}{{\mbox{sgn}}}

\newcommand{\px}{\partial_x}
\newcommand{\pt}{\partial_t}

\def\norm#1{\|#1\|}
\def\bra#1{\langle#1\rangle}
\def\wt#1{\widetilde{#1}}
\def\wh#1{\widehat{#1}}
\def\set#1{\{#1\}}

\begin{document}
\title[IBVP for the biharmonic NLS in a quarter plane]{Lower regularity solutions of the biharmonic Schr\"odinger equation in a quarter plane}
\author[Capistrano Filho]{Roberto de A. Capistrano Filho}
\address{\emph{Departmento de Matem\'atica,  Universidade Federal de Pernambuco (UFPE), 50740-545, Recife (PE), Brazil.}}
\email{roberto.capistranofilho@ufpe.br}
\author[Cavalcante]{M\'arcio Cavalcante}
\address{\emph{Instituto de Matem\'{a}tica, Universidade Federal de Alagoas (UFAL),  57072-900, Macei\'o (AL), Brazil.}}
\email{marcio.melo@im.ufal.br}
\author[Gallego]{Fernando A. Gallego}
\address{\emph{Departamento de Matematicas y Estad\'istica, Universidad Nacional de Colombia (UNAL), Cra 27 No. 64-60, 170003, Manizales, Colombia.}}
\email{fagallegor@unal.edu.co}



\begin{abstract}
This paper deals with the initial-boundary value problem of the biharmonic cubic nonlinear Schr\"odinger equation in a quarter plane with inhomogeneous Dirichlet-Neumann boundary data. We prove local well-posedness in the low regularity Sobolev spaces by introducing Duhamel boundary forcing operator associated to the linear equation in order  to construct solutions in the whole line. With this in hand, the energy and nonlinear estimates allow us to apply the Fourier restriction method, introduced by J. Bourgain, to get the main result of the article. Additionally, adaptations of this approach for the biharmonic cubic nonlinear Schr\"odinger equation on star graphs are discussed.

\end{abstract}
\maketitle

\section{Introduction}

\subsection{Presentation of the model}   Fourth-order nonlinear Schr\"odinger (4NLS) equation or biharmonic cubic nonlinear Schr\"odinger equation 
\begin{equation}
\label{fourtha}
i\partial_tu +\partial_x^2u-\partial_x^4u=\lambda |u|^2u,
\end{equation}
has been introduced by Karpman \cite{Karpman} and Karpman and Shagalov \cite{KarSha} to take into account the role of small fourth-order dispersion terms in the propagation of intense laser beams in a bulk medium with Kerr nonlinearity. Equation \eqref{fourtha} arises in many scientific fields such as quantum mechanics, nonlinear optics and plasma physics, and has been intensively studied with fruitful references (see \cite{Ben,CuiGuo,Karpman,Paus,Paus1} and references therein).

The past twenty years such 4NLS has been deeply studied from different mathematical viewpoint. For example, Fibich \textit{et al.} \cite{FiIlPa} worked various properties of the equation in the sub-critical regime, with part of their analysis relying on very interesting numerical developments.  The well-posedness and existence of solutions for different domains have been shown (see, for instance, \cite{CaCaGa,Kwak,Ozsari,Paus,Paus1,tsutsumi,Tzvetkov, WenChaiGuo1}) by means of the Fourier restriction method, energy method, forcing boundary operators, Laplace transform, harmonic analysis, Fokas method, etc.  

It is interesting to point out that there are many works related to the equation \eqref{fourtha} not only dealing with well-posedness theory. For example, recently Natali  and Pastor \cite{NataliPastor}, considered the fourth-order dispersive cubic nonlinear Schr\"odinger equation on the line with mixed dispersion. They proved the orbital stability, in the $H^2(\mathbb{R})$--energy space, by constructing a suitable Lyapunov function. Considering equation \eqref{fourtha} on the circle, Oh and Tzvetkov  \cite{Tzvetkov}, showed that the mean-zero Gaussian measures on Sobolev spaces $H^s (\mathbb{T})$, for $s>\frac{3}{4}$, are quasi-invariant under the flow. For instance, there has been a significant progress over the recent years and the reader can have a great view in \cite{Burq,Burq1} for the nonlinear Schr\"odinger equation.

In addition to these works, the first and the second authors worked recently with the purpose to prove controllability results for the 4NLS. More precisely, they proved that the solutions of the associated linear system \eqref{fourtha} is globally exponentially stable in a periodic domain $\mathbb{T}$,  by using certain properties of propagation of compactness and regularity in Bourgain spaces. Theses properties together with the local exact controllability ensure that fourth order nonlinear Schr\"odinger is globally exactly controllable, for details see \cite{CaCa}.

In a recently work, \"Ozsari and Yolcu \cite{Ozsari} proposed the equation \eqref{fourtha} without the term $\partial^2_xu$. This system has an interesting physical point of view, precisely, the model corresponds to a situation in which wave is generated from a fixed source such that it moves into the medium in one specific direction.

\subsection{Setting of the problem}

 We mainly consider the biharmonic Sch\"odinger equation on the right half-line 
\begin{equation}\label{fourth}
\begin{cases}
i\pt u - \px^4 u + \lambda |u|^2u=0, & (t,x)\in (0,T)  \times(0,\infty),\\
u(0,x)=u_0(x),                                   & x\in(0,\infty),\\
u(t,0)=f(t),\ u_x(t,0)=g(t)& t\in(0,T).
\end{cases}
\end{equation}
With suitable choices of $f(t)$ and $g(t)$ in the equation \eqref{fourth}, we are interested on the following initial-boundary value problem (IBVP):

\vspace{0.2cm}
\noindent \textit{Is the IBVP \eqref{fourth} local well-posed  in the low regularity Sobolev space, more precisely, in $H^s(\R^+)$ for $0\leq s<\frac12$?}
\vspace{0.2cm}

Before to present the answer for this question, let us to present some brief comments of the techniques to solve IBVPs on the half-line.


\subsection{Comments about the techniques to solve IBVPs on the half-line} Different techniques have been developed in the last years in order to solve IBVPs associated to some dispersive models on the half-line. In \cite{Fokas1}, Fokas introduced an approach to solve IBVPs associated to integrable nonlinear evolution equations, which is known as the unified transform method (UTM) or as Fokas transform method. The UTM provides a generalization of the Inverse Scattering Transform method from initial value problems (IVP) to IBVPs. The  classical method based on the Laplace transform was used successfully in the works \cite{Bona,BoShuZH,tzirakis,tzirakis3} developed by Bona, Sun, Zhang, Erdogan,  Compaan and Tzirakis.  Recently, a new approach was introduced by Colliander and Kenig \cite{CK} by recasting the IBVP on the half-line by a forced IVP defined in the line $\R$. To see other applications of this technique, we refer the results established by Holmer, Cavalcante and Corcho in the following works \cite{Cavalcante,CC,HolmerNLS,Holmerkdv}. On the other hand, in  \cite{Faminskii},  Faminskii used an approach based on the investigation of special solutions of a ``boundary potential'' type for solution of linearized Korteweg--de Vries (KdV) equation  in order to obtain global results for the IBVP associated to the KdV equation on the half-line with more general boundary conditions. More recently,  Fokas \textit{et al.} \cite{Fokas2} introduced a method which combines the UTM  with a contraction mapping principle.  We caution that this is only a small sample of the extant works on these techniques.

\subsection{Biharmonic NLS equation} As mentioned in the beginning of this introduction, 4NLS equation or biharmonic NLS equation 
\begin{equation}
\label{fourthA}
i\partial_tu +\partial_x^2u-\partial_x^4u=\lambda |u|^2u,
\end{equation}
was introduced by Karpman \cite{Karpman} and Karpman and Shagalov \cite{KarSha}.
Huo and Jia \cite{HuoJia} studied the Cauchy problem of one-dimensional fourth-order nonlinear Schr\"odinger equation  related to the vortex filament. They proved the local well-posedness for initial data in $H^s(\mathbb{R})$ for $s\geq\frac{1}{2}$ by using the  Fourier restriction norm method under certain coefficient condition. Concerning local well-posedness of the nonlinear fourth order Schr\"odinger equations, we cite \cite{Hao,Seg}. With respect of the global well-posedness, in one dimensional case with some restriction in the initial data for various nonlinearities, we infer \cite{Hay,Hay1,Hay2,Hay3} and, finally, for the study $n$-dimensional case the reader can see \cite{Beno,Beno1}.

Lastly, in a recent work of IBVP for biharmonic Schr\"odinger equation on the half-line
\begin{equation}
\label{fourthB}
i\partial_tu +\partial_x^4u=\lambda |u|^pu,
\end{equation}
 \"Ozsari and Yolcu \cite{Ozsari}, proved local well-posedness on the high regularity function spaces $H^s(\R^+)$, for $\frac12<s<\frac92$, with $s\neq\frac{3}{2}$. The authors used the Fokas method \cite{Fokas,Fokas1} combined with contraction arguments to achieve the result. 

\subsection{Main result} Now, let us present the main result of this article. Consider the biharmonic Sch\"odinger equation on the right half-line
\begin{equation}\label{fourthc}
\begin{cases}
i\pt u +\gamma\px^4 u + \lambda |u|^2u=0, & (t,x)\in (0,T)  \times(0,\infty),\\
u(0,x)=u_0(x),                                   & x\in(0,\infty),\\
u(t,0)=f(t),\ u_x(t,0)=g(t)& t\in(0,T),
\end{cases}
\end{equation}
for $\gamma, \lambda\in\mathbb{R}$. We say that system \eqref{fourthc} is focusing if $\gamma\lambda< 0$ and defocusing when $\gamma\lambda > 0$.
In this paper we will study the case when $\gamma=-1$, however the approach used here can be applied when $\gamma\in\mathbb{R}\setminus\{0\}$. 

The presence of two boundary conditions in \eqref{fourthc} can be  motivated by integral identities on smooth decaying solutions for the linear equation 
\begin{equation}\label{eq:linear 4kdv}
i\partial_t u-\partial_x^4u=0.
\end{equation}
Indeed, for a smooth decaying solution $u$ of \eqref{eq:linear 4kdv} and $T>0$, we have 
\begin{equation}\label{unico1}
\begin{split}
	\int_0^{\infty}|u(T,x)|^2dx=&\int_0^{\infty}|u(0,x)|^2dx\\
	& -  \int_0^T\text{Im}(\partial_x^3u(t,0)
\overline{u}(t,0)) dt +\int_0^T\text{Im}(\partial_x^2u(t,0)\partial_x\overline{u}(t,0)) dt.
	\end{split}
\end{equation}
Thus, from \eqref{unico1} we can conclude  that if we assume $u(0,x)=u(t,0)=u_x(t,0)=0$ the linear solution for \eqref{eq:linear 4kdv} is the trivial one.

It is well-known by Kenig \textit{et al.} \cite{KPV1991} that the \textit{local smoothing effect} for the fourth-order linear group operator $e^{it\px^4}$ 
\begin{equation}\label{eq:lsm}
\|\partial_x^je^{it\partial_x^4}\phi\|_{L_x^{\infty}\dot{H}^{\frac{2s+3-2j}{8}}(\mathbb{R}_t)}\leq c\|\phi\|_{H^{s}(\mathbb{R})},\ \text{for}\ j=0,1\ \text{and} \ s\in \R,
\end{equation}
which motivates the relation of regularities among initial and boundary data. 

Thus, we are able to present the main goal in the paper: To answer the problem cited on the beginning of this introduction, that is, to show the local well-posedness of \eqref{fourthc} in the low regularity Sobolev space $H^s(\R^+)$, for $0\leq s<\frac12$.

We state the main theorem for IBVP \eqref{fourthc} as follows.
\begin{theorem}\label{theorem1}
	Let $s \in [0,\frac12)$. For given initial-boundary data $$(u_0,f,g)\in H^s(\mathbb{R}^+)\times H^{\frac{2s+3}{8}}(\mathbb{R}^+)\times H^{\frac{2s+1}{8}}(\mathbb{R}^+),$$ there exist a positive time $$T:=T\left(\|u_0\|_{H^s(\mathbb{R}^+)},\|f\|_{H^{\frac{2s+3}{8}}(\mathbb{R}^+)},\|g\|_{H^{\frac{2s+1}{8}}(\mathbb{R}^+)}\right),$$ and unique solution $u(t,x) \in C((0 , T);H^s(\R^+))$ of the IBVP \eqref{fourthc}, when $\gamma=-1$, satisfying
\[u \in C\bigl(\mathbb{R}^+;\; H^{\frac{2s+3}{8}}(0,T)\bigr) \cap X^{s,b}((0,T) \times \R^+) \; \mbox{ and } \; \partial_xu\in C\bigl(\R^+;\; H^{\frac{2s+1}{8}}(0,T)\bigr),\]
for some $b(s) < \frac12$.  Moreover, the map $(u_0,f,g)\longmapsto u$ is analytic from $H^s(\mathbb{R}^+)\times H^{\frac{2s+3}{8}}(\mathbb{R}^+)\times H^{\frac{2s+1}{8}}(\mathbb{R}^+)$ to  $C\big((0,T);\,H^s(\mathbb{R}^+)\big)$.
\end{theorem}

\begin{remarks}Finally, the following comments are now given in order:
\begin{itemize}
\item[1.]The proof of Theorem \ref{theorem1} is based on the Fourier restriction method for a suitable extension of solutions. We first convert the IBVP of \eqref{fourthc} posed in $\R^+ \times \R^+$ to the initial value problem (IVP) of \eqref{fourthc} (integral equation formula) in the whole space $\R \times \R$ (see Section \ref{sec:Duhamel boundary forcing operator}) by using the Duhamel boundary forcing operator. The energy and nonlinear estimates (will be established in Sections \ref{sec:energy}) allow us to apply the Picard iteration method for IVP of \eqref{fourthc}, and hence we can complete the proof. The new tools used here are the Duhamel boundary forcing operator for the fourth-order linear equation and its analysis.
\item[2.]Note that Theorem \ref{theorem1} give us the local well-posedness in low regularity for the biharmonic nonlinear Schr\"odinger equation. However, in \cite{Ozsari},  the authors showed the local well-posedness in the Sobolev spaces, by using Fokas approach. We point out that the low regularity in our   main result is obtained using the boundary forcing operator, proposed by Holmer, which has been obtained in an independent way and with other approach to the one proved in \cite{Ozsari}.
\item[3.] The approach used in our result, together with some extension as it was done in \cite{Cavalcante,CaK,HolmerNLS,Holmerkdv} also guarantee the local well-posedness result in high regularity.
\end{itemize}
\end{remarks}

\subsection{Notations} In all this paper, we will consider  $\R^+$ as $(0,\infty)$. Moreover, for positive real numbers $x,y \in \R^+$, we mean $x \lesssim y$ by $x \le Cy$ for some $C>0$. Also, denote $x \sim y$ by $x \lesssim y$ and $y\lesssim x$. Similarly, $\lesssim_s$ and $\sim_s$ can be defined, where the implicit constants depend on $s$. 

\vspace{0.2cm}

Our work is outlined in the following way: In Section \ref{sec:pre}, we introduce some function spaces defined on the half-line and construct the solution spaces. Section \ref{sec:Duhamel boundary forcing operator} is devoted to introduce the boundary forcing operator for the biharmonic Schr\"odinger equation. In section \ref{sec:energy}, we show the energy estimates and present the trilinear estimates, respectively. The main result of this article, Theorem \ref{theorem1}, is proved in Section \ref{sec:main proof 1}. Finally, Section \ref{further}, we present some open problems which seems to be of interest from the mathematical point of view.

\section{Preliminaries}\label{sec:pre}

Throughout the paper, we fix a cut-off function $\psi(t):=\psi$
\begin{equation}\label{eq:cutoff}
\psi \in C_0^{\infty}(\mathbb{R}) \quad \mbox{such that} \quad 0 \le \psi \le1, \quad  \psi \equiv 1 \; \mbox{ on } \; [0,1], \quad \psi \equiv 0, \;\text{for } |t| \ge 2, 
\end{equation}
and for $T>0$ we denote $\psi_{T}(t)=\frac{1}{T}\psi(\frac{t}{T})$.

\subsection{Sobolev spaces on the half-line}
For $s\geq 0$, we define the homogeneous $L^2$-based Sobolev spaces $\dot{H}^s=\dot H^s(\R)$  by the norm $\|\phi\|_{\dot{H}^s}=\||\xi|^s\hat{\psi}(\xi)\|_{L^2_{\xi}}$ and the $L^2$-based inhomogeneous Sobolev spaces $H^s=H^s(\mathbb{R})$ by the norm $\|\phi\|_{{H}^s}=\|(1+|\xi|^2)^{\frac{s}{2}}\hat{\psi}(\xi)\|_{L^2_{\xi}}$, where $\hat{\phi}$ denotes the Fourier transform of $\phi$. Moreover, we say that $f \in H^s(\mathbb{R}^+)$, if there exists $F \in H^s(\R)$ such that $f(x)=F(x)$ for $x>0$, in this case we set $$\|f\|_{H^s(\mathbb{R}^+)}=\inf_{F}\|F\|_{H^{s}(\mathbb{R})}.$$ On the other hand, for $s \in \R$, $f \in H_0^s(\mathbb{R}^+)$ provided that there exists $F \in H^s(\R)$ such that $F$ is the extension of $f$ on $\R$ and $F(x) = 0$ for $x<0$. In this case, we set $\norm{f}_{H^s_0(\R^+)} = \inf_{F} \norm{F}_{H^s(\R)}$. 
For $s<0$, we define $H^s(\mathbb{R}^+)$ as the dual space of $H_0^{-s}(\mathbb{R}^+)$.

Let us also define the sets $C_0^{\infty}(\mathbb{R}^+)=\{f\in C^{\infty}(\mathbb{R});\, \supp f \subset [0,\infty)\}$ and $C_{0,c}^{\infty}(\mathbb{R}^+)$ as the subset of $C_0^{\infty}(\mathbb{R}^+)$, whose members have a compact support on $(0,\infty)$. We remark that $C_{0,c}^{\infty}(\mathbb{R}^+)$ is dense in $H_0^s(\mathbb{R}^+)$ for all $s\in \mathbb{R}$.

To finish this subsection we will give some elementary properties of the Sobolev spaces.

\begin{lemma} \cite[Lemma 3.5]{JK1995}\label{sobolevh0}
For $-\frac{1}{2}<s<\frac{1}{2}$ and $f\in H^s(\mathbb{R})$, we have
	\begin{equation}\label{eq:0}
	\|\chi_{(0,\infty)}f\|_{H^s(\mathbb{R})}\leq c \|f\|_{H^s(\mathbb{R})}.
	\end{equation} 
\end{lemma}

\begin{lemma}\cite[Lemma 2.8]{CK}\label{sobolev0}
If $0\leq s<\frac{1}{2}$, then, for the cut-off function $\psi$ defined in \eqref{eq:cutoff}, $\|\psi f\|_{H^s(\mathbb{R})}\leq c \|f\|_{\dot{H}^{s}(\mathbb{R})}$ and $\|\psi f\|_{\dot{H}^{-s}(\mathbb{R})}\leq c \|f\|_{H^{-s}(\mathbb{R})}$ , where the constant $c$ depends only on $s$ and $\psi$. 
\end{lemma}
\begin{remark} Lemma \ref{sobolev0} is equivalent to $$\|f\|_{H^s(\mathbb{R})} \sim \|f\|_{\dot{H}^{s}(\mathbb{R})},$$ for $-\frac12 < s < \frac12$ where $f \in H^s(\mathbb{R})$ with $\supp f \subset [0,1]$.
\end{remark}

The following two auxiliaries lemmas can be found in \cite{CK} and its proofs will be omitted.

\begin{lemma}\cite[Proposition 2.4]{CK}\label{alta}
	If $\frac{1}{2}<s<\frac{3}{2}$ the following statements are valid:
	\begin{enumerate}
\item [(a)] $H_0^s(\R^+)=\big\{f\in H^s(\R^+);f(0)=0\big\},$\medskip
\item [(b)] If  $f\in H^s(\R^+)$ with $f(0)=0$, then $\|\chi_{(0,\infty)}f\|_{H_0^s(\R^+)}\leq c \|f\|_{H^s(\R^+)}$.
	\end{enumerate}
\end{lemma}

\begin{lemma}\cite[Proposition 2.5]{CK}\label{cut}
	Let $-\infty<s<\infty$ and $f\in  H_0^s(\mathbb{R}^+)$. For the cut-off function $\psi$ defined in \eqref{eq:cutoff}, we have $ \|\psi f\|_{H_0^s(\mathbb{R}^+)}\leq c \|f\|_{H_0^s(\mathbb{R}^+)}$.
\end{lemma}

\subsection{Solution spaces}\label{sec:sol space}
For $f \in \Sch (\R^2)$, we denote by $\wt{f}$ or $\ft (f)$ the Fourier transform of $f$ with respect to both spatial and time variables
\[\wt{f}(\tau , \xi)=\int _{\R ^2} e^{-ix\xi}e^{-it\tau}f(t,x) \;dxdt .\]
Moreover, we use $\ft_x$ and $\ft_t$ to denote the Fourier transform with respect to space and time variable respectively (also we use $\,\, \wh{\;}\,\,$ for both cases). 

In a celebrated paper in 90's, Bourgain \cite{Bourgain1993} established a way to prove the well-posedness of a classes of dispersive systems. More precisely, on the Sobolev spaces $H^s$, for smaller values of s, Bourgain  found a yet more suitable smoothing property for solutions of the Korteweg-de Vries equation. 

In this spirit, for $s,b\in \mathbb{R}$, we introduce the classical Bourgain spaces $X^{s,b}$ associated to \eqref{fourth} as the completion of $\Sch'(\mathbb{R}^2)$ under the norm
\[\norm{f}_{X^{s,b}}^2 = \int_{\R^2} \bra{\xi}^{2s}\bra{\tau +\xi^4}^{2b}|\wt{f}(\tau,\xi)|^2 \; d\xi d\tau, \]
where $\bra{\cdot} = (1+|\cdot|^2)^{1/2}$.

One of the basic property of $X^{s,b}$ can be read as follows:
\begin{lemma}\cite[Lemma 2.11 ]{Tao2006}\label{lem:Xsb}
Let $\psi(t)$ be a Schwartz function in time. Then, we have 
\[\norm{\psi(t)f}_{X^{s,b}} \lesssim_{\psi,b} \norm{f}_{X^{s,b}}.\]
\end{lemma}

Ginibre \textit{et al.} \cite{GTV} while establishing local well-posedness results for the Zakharov system showed the following important estimate:
\begin{lemma}\label{gvt}
	Let $-\frac{1}{2}< b'< b\leq 0$\, or\, $0\leq b'<b<\frac{1}{2}$, $w\in X^{s,b}(\phi)$ and $s\in \mathbb{R}$. Then
	\begin{equation*}
	\|\psi_{T}w\|_{ X^{s,b'}(\phi)}\leq c T^{b-b'}\|w\|_{ X^{s,b}(\phi)}.
	\end{equation*}
\end{lemma}

As is well-known, the space $X^{s,b}$  with $b> \frac12$ is well-adapted to study the IVP of dispersive equations. However, in the study of IBVP, the standard argument cannot be applied directly. This is due to the lack of hidden  regularity, more precisely, the control of (derivatives) time trace norms of the Duhamel boundary operator  requires  to work in $X^{s,b}$-type spaces for $b<\frac12$, since the full regularity range cannot be covered (see Lemma \ref{duhamel} inequality \eqref{derivative1}). 

Therefore, to treat the solution of our problem, set the solution space denoted by $Z^{s,b}$ with the following norm
\[\norm{f}_{Z^{s,b}(\R^2) }= \sup_{t \in \R} \norm{f(t,\cdot)}_{H^s(\R)} + \sum_{j=0}^{1}\sup_{x \in \R} \norm{\px^jf(\cdot,x)}_{H^{\frac{2s+3-2j}{8}}(\R)} + \norm{f}_{X^{s,b}} .\]
 The spatial and time restricted space of $Z^{s,b}(\R^2)$ is defined by the standard way:
\[Z^{s,b}((0 , T)\times \R^+) = Z^{s,b} \Big|_{(0 , T)\times \R^+}\]
equipped with the norm
\[\norm{f}_{Z^{s,b}((0 , T)\times \R^+) } = \inf\limits_{g \in Z^{s,b}} \set{\norm{g}_{Z^{s,b}} : g(t,x) = f(t,x) \; \mbox{ on } \; (0,T) \times \R^+}. \]

\subsection{Riemann-Liouville fractional integral} Before we begin our study of the IBVP for \eqref{fourth}, we just give a brief summary of the Riemann-Liouville fractional integral operator, the reader can see \cite{CK, Holmerkdv} for more details. 

Let us define the function $t_+$ as follows
\[t_+ = t \quad \mbox{if} \quad t > 0, \qquad t_+ = 0  \quad \mbox{if} \quad t \le 0.\]
The tempered distribution $\frac{t_+^{\alpha-1}}{\Gamma(\alpha)}$ is defined like a locally integrable function for Re $\alpha>0$ by
\begin{equation*}
	\left \langle \frac{t_+^{\alpha-1}}{\Gamma(\alpha)},\ f \right \rangle=\frac{1}{\Gamma(\alpha)}\int_0^{\infty} t^{\alpha-1}f(t)dt.
\end{equation*}
It follows that
\begin{equation}\label{gamma}
	\frac{t_+^{\alpha-1}}{\Gamma(\alpha)}=\partial_t^k\left( \frac{t_+^{\alpha+k-1}}{\Gamma(\alpha+k)}\right),
\end{equation}
for all $k\in\mathbb{N}$. Expression  \eqref{gamma} can be used to extend the definition of $\frac{t_+^{\alpha-1}}{\Gamma(\alpha)}$, $\forall \alpha \in \mathbb{C}$ in the sense of distributions. In fact, a change of contour shows that the Fourier transform of $\frac{t_+^{\alpha-1}}{\Gamma(\alpha)}$ is the following one
\begin{equation}\label{transformada}
	\left(\frac{t_+^{\alpha-1}}{\Gamma(\alpha)}\right)^{\widehat{}}(\tau)=e^{-\frac{1}{2}\pi i \alpha}(\tau-i0)^{-\alpha},
\end{equation}
where \begin{equation}\label{transformada2}
(\tau-i0)^{-\alpha} = |\tau|^{-\alpha}\chi_{(0,\infty)}+e^{\alpha \pi i}|\tau|^{-\alpha}\chi_{(-\infty,0)}
\end{equation} 
is the distributional limit. For $\alpha \notin \Z$, by using \eqref{transformada2}, we rewrite \eqref{transformada} in the following way
\begin{equation}\label{transformada1}
	\left(\frac{t_+^{\alpha-1}}{\Gamma(\alpha)}\right)^{\widehat{}}(\tau)=e^{-\frac12 \alpha \pi i}|\tau|^{-\alpha}\chi_{(0,\infty)}+e^{\frac12 \alpha \pi i}|\tau|^{-\alpha}\chi_{(-\infty,0)}.
\end{equation}

For $f\in C_0^{\infty}(\mathbb{R}^+)$, define $	\mathcal{I}_{\alpha}f$ as
\begin{equation*}
	\mathcal{I}_{\alpha}f=\frac{t_+^{\alpha-1}}{\Gamma(\alpha)}*f.
\end{equation*}
Thus, for $\mbox{Re }\alpha>0$, we have
\begin{equation}\label{eq:IO}
	\mathcal{I}_{\alpha}f(t)=\frac{1}{\Gamma(\alpha)}\int_0^t(t-s)^{\alpha-1}f(s) \; ds.
\end{equation}
The following properties easily hold 
\begin{align*}
\mathcal{I}_0f=f, \quad \mathcal{I}_1f(t)=\int_0^tf(s) \; ds, \quad  \mathcal{I}_{-1}f=f' \quad \text{and } \quad \mathcal{I}_{\alpha}\mathcal{I}_{\beta}=\mathcal{I}_{\alpha+\beta}.
\end{align*}
Moreover, the lemmas below can be found in \cite{Holmerkdv}, and we will omit the proofs.
\begin{lemma}\cite[Lemma 2.1]{Holmerkdv}
	If $f\in C_0^{\infty}(\mathbb{R}^+)$, then $\mathcal{I}_{\alpha}f\in C_0^{\infty}(\mathbb{R}^+)$, for all $\alpha \in \mathbb{C}$.
\end{lemma}
\begin{lemma}\cite[Lemma 5.3]{Holmerkdv}\label{lio}
	If $0\leq \mathrm{Re} \ \alpha <\infty$ and $s\in \mathbb{R}$, then $\|\mathcal{I}_{-\alpha}h\|_{H_0^s(\mathbb{R}^+)}\leq c \|h\|_{H_0^{s+\alpha}(\mathbb{R}^+)}$, where $c=c(\alpha)$.
\end{lemma}
\begin{lemma}\cite[Lemma 5.4]{Holmerkdv}
	If $0\leq \mathrm{Re}\ \alpha <\infty$, $s\in \mathbb{R}$ and $\mu\in C_0^{\infty}(\mathbb{R})$, then
	$\|\mu\mathcal{I}_{\alpha}h\|_{H_0^s(\mathbb{R}^+)}\leq c \|h\|_{H_0^{s-\alpha}(\mathbb{R}^+)},$ where $c=c(\mu, \alpha)$.
\end{lemma}

\subsection{Oscillatory integral} In this subsection, we will define the oscillatory integral which one is the key to define, in the next section, the Duhamel boundary forcing operator. Let 
\begin{equation}\label{eq:oscil}
B(x)=\frac{1}{2\pi}\int_{\R}e^{ix\xi}e^{-i \xi^4} \; d\xi.
\end{equation} 
We first calculate $B(0)$. A change of variable ($\eta = \xi^4$), give us the following
\begin{align*}
B(0) &=\frac{1}{2\pi}\int_{\R}e^{-i \xi^4} \;d\xi=\frac{1}{4\pi}\int_{0}^{+\infty}e^{-i \eta}\eta^{-\frac34}\;d\eta.
\end{align*}
Now, a change of contour yields
\begin{align*}
B(0) &= \frac{(-i)^{1-3/4}}{4\pi}\int_{0}^{+\infty}e^{-t}t^{\frac{1}{4}-1}\;dt  =\frac{(-i)^{1/4}}{4\pi}\Gamma\left( \frac14 \right) =-\frac{i^{7/4}}{\pi}\Gamma\left( \frac54 \right). 
\end{align*}


Let us obtain the Mellin transform of $B(x)$.
\begin{lemma} \label{mellin} For Re $\lambda>0$ we have
		\begin{equation}\label{mellin1}
		\int_0^{\infty}x^{\lambda-1}B(x)dx=\frac{\Gamma(\lambda)\Gamma\left(\frac14-\frac{\lambda}{4}\right)}{8\pi}\left(e^{-i\frac{\pi}{8}(1+3\lambda)}+e^{-i\frac{\pi}{8}(1-5\lambda)}\right).
		\end{equation}
\end{lemma}
\begin{proof}
By analyticity argument, we can assume that $\lambda$ is a real number in the set $(0,3/8)$.

Consider $$B_1(x)=\frac{1}{2\pi}\int_{0}^{+\infty}e^{i x \xi}e^{-i \xi^4}d\xi$$ and $$B_2(x)=\frac{1}{2\pi}\int_{-\infty}^{0}e^{i x \xi}e^{-i \xi^4}d\xi=\frac{1}{2\pi}\int_0^{\infty}e^{-ix\xi}e^{-i\xi^4}d\xi,$$ then we have $B(x)= B_1(x)+B_2(x)$. Define $$B_{1,\epsilon}(x)=\frac{1}{2\pi}\int_{0}^{+\infty}e^{i x \xi}e^{-i \xi^4}e^{-\epsilon \xi}d\xi.$$ By using dominated convergence theorem and Fubini's theorem we have
\begin{equation}
\begin{split}
\int_{0}^{\infty}x^{\lambda-1}B_1(x)dx&=\lim_{\epsilon \rightarrow 0}\lim_{\delta \rightarrow 0}\int_{0}^{\infty}e^{-\delta x}x^{\lambda-1}B_{1,\epsilon}(x)dx\\
& =\lim_{\epsilon \rightarrow 0}\lim_{\delta \rightarrow 0} \frac{1}{2\pi}\int_{0}^{+\infty}e^{-i \xi^4}e^{-\epsilon \xi} \int_{0}^{\infty}e^{i x \xi}e^{-\delta x}x^{\lambda-1}dx d\xi.
\end{split}
\end{equation}
Using a change of contour, we get that
\begin{equation}
\begin{split}
\int_{0}^{+\infty}e^{i x \xi}e^{-\delta x}x^{\lambda-1}dx=\xi^{-\lambda }e^{i \lambda \frac{\pi}{2}}\Gamma\left(\lambda, -\frac{\delta}{\xi}\right),
\end{split}
\end{equation}
where $\Gamma(\lambda,z)=\int_0^{+\infty}r^{\lambda-1}e^{irz}e^{-r}dr$.
Again, thanks to dominated convergence theorem it follows that
\begin{equation}
\lim_{\delta \rightarrow 0} \int_0^{+\infty}e^{i x \xi}e^{-\delta x}x^{\lambda-1}dx=\xi^{-\lambda }e^{i \lambda \frac{\pi}{2}}\Gamma\left( \lambda\right).
\end{equation}
Once more applying dominated convergence theorem and changing the contour we conclude that
\begin{equation}
\begin{split}
\int_{0}^{+\infty}x^{\lambda-1}B_1(x)dx&=\frac{\Gamma(\lambda)}{2\pi}e^{i\lambda\frac{\pi}{2}}\lim_{\varepsilon \rightarrow 0}\int_0^{+\infty}e^{-i \xi^4}e^{-\epsilon \xi} \xi^{-\lambda}d\xi\\
&=\frac{\Gamma(\lambda)}{2\pi}e^{i\lambda\frac{\pi}{2}}\frac{1}{4}\lim_{\varepsilon \rightarrow 0}\int_{0}^{+\infty}e^{-i\eta}e^{-\epsilon \eta^{1/4} }(\eta)^{-\frac{\lambda+3}{4}}d\eta\\
&=\frac{\Gamma(\lambda)}{2\pi}e^{i\lambda\frac{\pi}{2}}\frac{1}{4}e^{-\frac{\pi i}{2}(\frac{1-\lambda}{4})}\Gamma\left(\frac14-\frac{\lambda}{4}\right)\\
&=\frac{\Gamma(\lambda)\Gamma\left(\frac14-\frac{\lambda}{4}\right)}{8\pi}e^{-i\frac{\pi}{8}(1-5\lambda)}.
\end{split}
\end{equation}

In the similar way, by using the following identity
\begin{equation}
\begin{split}
	\
	\int_{0}^{+\infty}e^{-i x \xi}e^{-\delta x}x^{\lambda-1}dx=\xi^{-\lambda }e^{-i \lambda \frac{\pi}{2}}\Gamma\left(\lambda, \frac{\delta}{\xi}\right),
\end{split}
\end{equation}
we obtain
\begin{equation}
\begin{split}
\int_{0}^{+\infty}x^{\lambda-1}B_2(x)dx&=\frac{\Gamma(\lambda)}{2\pi}e^{-i\lambda\frac{\pi}{2}}\frac{1}{4}e^{-\frac{\pi i}{2}(\frac{1-\lambda}{4})}\Gamma\left(\frac14-\frac{\lambda}{4}\right)\\
&=\frac{\Gamma(\lambda)\Gamma\left(\frac14-\frac{\lambda}{4}\right)}{8\pi}e^{-i\frac{\pi}{8}(1+3\lambda)}.
\end{split}
\end{equation}
Finally, as we can split by $B(x)= B_1(x)+B_2(x)$,  \eqref{mellin1} holds.
\end{proof}

\section{Duhamel boundary forcing operator}\label{sec:Duhamel boundary forcing operator}
In this section, we study the Duhamel boundary forcing operator, which was introduced by Colliander and Kenig \cite{CK}, in order to construct the solution to \eqref{fourth} forced by boundary conditions. We also refer the following works \cite{Cavalcante,CC,HolmerNLS} for a well exposition about this topic.
\subsection{Duhamel boundary forcing operator class}

Let us introduce the Duhamel boundary forcing operator associated to the linearized biharmonic Schr\"odinger equation. Consider
\begin{equation}\label{eq:M}
M = \frac{1}{ B(0)\Gamma(3/4)}.
\end{equation} 
For $f\in C_0^{\infty}(\mathbb{R}^+)$, define the boundary forcing operator $\mathcal{L}^0$ (of order $0$) as
\begin{equation}\label{eq:BFO}	
\mathcal{L}^0f(t,x):=M\int_0^te^{i(t-t')\partial_x^4}\delta_0(x)\mathcal{I}_{-\frac34}f(t')dt',
\end{equation}
where $e^{it\partial_x^4}$ denotes the group associated to \eqref{eq:linear 4kdv} given by 
\begin{align*}
e^{it\partial_x^4}\psi(x)=\frac{1}{2\pi}\int_{\mathbb{R}} e^{i x \xi}e^{-it\xi^4}\hat{\psi}(\xi)d\xi.
\end{align*}
Note that the property of convolution operator ($\partial_x (f * g) = (\partial_xf) * g = f * (\partial_xg)$) and the integration by parts in $t'$ of \eqref{eq:BFO} yield that
\begin{equation}\label{eq:BFO1}
i\mathcal{L}^0(\partial_t f)(t,x) = iM\delta_0(x)\mathcal{I}_{-\frac34}f(t) + \partial_x^4\mathcal{L}^0f(t,x).
\end{equation}

By a change of variable and using \eqref{eq:oscil}, we get that
\begin{equation}\label{forcing}
\begin{split}
	\mathcal{L}^0f(t,x)&=M\int_0^te^{i(t-t')\partial_x^4}\delta_0(x)\mathcal{I}_{-\frac34}f(t')dt'\\
	&= M \int_0^t B\left(\frac{x}{(t-t')^{1/4}}\right)\frac{\mathcal{I}_{-\frac34}f(t')}{(t-t')^{1/4}}dt'.
\end{split}	
\end{equation}
We are now in position to make it precise when the boundary forcing term is continuous or discontinuous. More precisely, the following lemma holds.

\begin{lemma}\label{continuity}
	Let $f\in C_{0,c}^{\infty}(\mathbb{R}^+)$.
	\begin{itemize}
\item[(a)] For fixed $ 0 \le t \le 1$, $\partial_x^k \mathcal{L}^0f(t,x)$, $k=0,1,2$, is continuous in $x \in \mathbb{R}$ and has the decay property in terms of the spatial variable as follows:
\begin{equation}\label{eq:decay1}
|\partial_x^k \mathcal{L}^0f(t,x)| \lesssim_{N} \norm{f}_{H^{N+k}}\bra{x}^{-N}, \quad N \ge 0.
\end{equation}
\item[(b)] For fixed $ 0 \le t \le 1$, $\partial_x^3\mathcal{L}^0f(t,x)$ is continuous in $x$ for $x\neq 0$ and it is discontinuous at $x =0$ satisfying
\[\lim_{x\rightarrow 0^{-}}\partial_x^3\mathcal{L}^0f(t,x)=-i\frac{M}{2}\mathcal{I}_{-3/4}f(t),\ \lim_{x\rightarrow 0^{+}}\partial_x^3\mathcal{L}^0f(t,x)=i\frac{M}{2}\mathcal{I}_{-3/4}f(t).\] $\partial_x^3\mathcal{L}^0f(t,x)$ also has the decay property in terms of the spatial variable
\begin{equation}\label{eq:decay2}
|\partial_x^3\mathcal{L}^0f(t,x)| \lesssim_{N} \norm{f}_{H^{N+3}}\bra{x}^{-N}, \quad N \ge 0.
\end{equation}
	\end{itemize}
\end{lemma}

\begin{proof}
In fact, the continuity of  $\partial_x^k \mathcal{L}^0f(t,x)$ follows from \eqref{forcing}, for $k=0,1,2$ and the proof of \eqref{eq:decay1} exactly follows the idea  introduced by Holmer in  \cite[Lemma 12]{HolmerNLS}.  Moreover, \eqref{eq:decay1} and \eqref{eq:BFO1} yield that $\partial_x^4\mathcal{L}^0f(t,x)$ is discontinuous only at $x =0$ of size $M\mathcal{I}_{-\frac34}f(t)$ (where $M$ is defined as \eqref{eq:M}), and the decay bounds \eqref{eq:decay2} holds.  
	\end{proof}
	\begin{remark} Lemma \ref{continuity} ensures that $\mathcal{L}^0f(t,0) = f(t)$.
	\end{remark}
We are now in position to generalize the boundary forcing operator \eqref{eq:BFO}. Firstly, for $\mbox{Re}\ \lambda > -4$ and given $g\in C_0^{\infty}(\mathbb{R}^+)$, we define
\begin{equation}\label{eq:BFOC}
\mathcal{L}^{\lambda}g(t,x)=\left[\frac{x_{-}^{\lambda-1}}{\Gamma(\lambda)}*\mathcal{L}^0\big(\mathcal{I}_{-\frac{\lambda}{4}}g\big)(t, \cdot)   \right](x),
\end{equation}
where $*$ denotes the convolution operator and $\frac{x_{-}^{\lambda-1}}{\Gamma(\lambda)}=\frac{(-x)_+^{\lambda-1}}{\Gamma(\lambda)}$.  In particular, for Re $\lambda>0$, we have
 \begin{equation}\label{forcing2}
\mathcal{L}^{\lambda}g(t,x)=\frac{1}{\Gamma(\lambda)}\int_{x}^{\infty}(y-x)^{\lambda-1}\mathcal{L}^0\big(\mathcal{I}_{-\frac{\lambda}{4}}g\big)(t,y)dy.
\end{equation}
Property of convolution operator ($\partial_x^4 (f * g) = (\partial_x^4f) * g = f * (\partial_x^4g)$) and \eqref{eq:BFO1} give us
\begin{equation}\label{classe22}
\begin{aligned}
\mathcal{L}^{\lambda}g(t,x)&=\left[\frac{x_{-}^{(\lambda+4)-1}}{\Gamma(\lambda+4)}*\px^4\mathcal{L}^0\big(\mathcal{I}_{-\frac{\lambda}{4}}g\big)(t, \cdot)   \right](x)\\
&=iM\frac{x_{-}^{(\lambda+4)-1}}{\Gamma(\lambda+4)}\mathcal{I}_{-\frac{3}{4}-\frac{\lambda}{4}}g(t)+i\int_{x}^{\infty}\frac{(y-x)^{(\lambda+4)-1}}{\Gamma(\lambda+4)}\mathcal{L}^0\big(\partial_t\mathcal{I}_{-\frac{\lambda}{4}}g\big)(t,y)dy,
\end{aligned}
\end{equation}
for $\mbox{Re}\  \lambda > -4$, where $M$ is defined as \eqref{eq:M}.  From \eqref{eq:BFO1} and \eqref{eq:BFOC}, we have
\begin{equation*}
(i\partial_t-\partial_x^4)\mathcal{L}^{\lambda}g(t,x)=iM\frac{x_{-}^{\lambda-1}}{\Gamma(\lambda)}\mathcal{I}_{-\frac{3}{4}-\frac{\lambda}{4}}g(t),
\end{equation*}\color{black}
in the distributional sense. 

To finish this subsection, we will give two lemmas concerning  the spatial continuity and decay properties of the $\mathcal{L}^{\lambda}g(t,x)$ and the explicit values for $\mathcal{L}^{\lambda}f(t,0)$, respectively. 
\begin{lemma}\label{holmer1}
	Let $g\in C_0^{\infty}(\mathbb{R}^+)$ and $M$ be as in \eqref{eq:M}. Then, we have
	\begin{equation}\label{eq:relation}
\mathcal{L}^{-k}g=\partial_x^k\mathcal{L}^{0}\mathcal{I}_{\frac{k}{4}}g, \qquad k=0,1,2,3.
	\end{equation}
	Moreover, $\mathcal{L}^{-3}g(t,x)$ is continuous in $x \in \R \setminus \set{0}$ and has a step discontinuity  at $x=0$. For real $\lambda$ satisfying $\lambda>-3$, $\mathcal{L}^{\lambda}g(t,x)$ is continuous in $x \in\mathbb{R}$. For $-3\leq\lambda\leq 1$ and  $0\leq t\leq 1$, $\mathcal{L}^{\lambda}g(t,x)$ satisfies the following decay bounds:
	\begin{align*}
	&|\mathcal{L}^{\lambda}g(t,x)|\leq c_{\lambda,g}\langle x\rangle^{\lambda-1}, \quad \text{for all} \quad  x\geq 0\\
	\intertext{and}
	&|\mathcal{L}^{\lambda}g(t,x)|\leq c_{m,\lambda,g}\langle x\rangle^{-m}, \quad \text{for all} \quad x\geq 0 \quad \text{and} \quad m\geq0.
	\end{align*}
\end{lemma}

\begin{proof}
We give below the sketch of the proof. The detailed argument can be found in \cite{Holmerkdv}. By using \eqref{classe22}, we have  that \eqref{eq:relation} follows. Moreover, Lemma \ref{continuity} together with \eqref{eq:relation} guarantee the continuity (except for $x = 0$ when $\lambda = -3$) and discontinuity at $x =0$ of $\mathcal{L}^{\lambda}g$ for $\lambda \ge -3$ and $\lambda = -3$, respectively. The proof of decay bounds can be obtained by using \eqref{classe22}, \eqref{eq:BFO1} and Lemma \ref{continuity}. 
\end{proof}
\begin{lemma}\label{trace1}
	 For $\mbox{Re}\ \lambda>  -4$ and $f\in C_0^{\infty}(\R^+)$, we have the following value of $\mathcal{L}^{\lambda}f(t,0)$:
	\begin{equation}\label{lr0}
	\mathcal{L}^{\lambda}f(t,0)=\frac{M}{8\ }f(t)\left(\frac{e^{-i\frac{\pi}{8}(1+3\lambda)}+e^{-i\frac{\pi}{8}(1-5\lambda)}}{\sin(\frac{1-\lambda}{4}\pi)}\right).
	\end{equation}
	\end{lemma}

\begin{proof}
By using formula \eqref{classe22} we get
	\[\mathcal{L}^{\lambda}f(t,0)
	=i\int_{0}^{\infty}\frac{y^{(\lambda+4)-1}}{\Gamma(\lambda+4)}\mathcal{L}^0\big(\pt \mathcal{I}_{-\frac{\lambda}{4}}f\big)(t,y)dy.\]
This show that	 $	\mathcal{L}^{\lambda}f(t,0)$  is analytic, in $\lambda$, for $\mbox{Re}\ \lambda> -4$.

By analyticity argument, it suffices to consider the case when $\lambda$ is a positive real number and \eqref{forcing}, where $M$ is defined as in \eqref{eq:M}. In fact, in order to use \eqref{mellin1}, we  take $\lambda \in (0,3/8)$ in \eqref{forcing2}. Thus, in the calculations,  we use the representation \eqref{forcing2} for $\lambda > 0$. Fubini's Theorem, the change of variable, \eqref{mellin} and \eqref{eq:IO}, yield that
	\[\begin{split}
	\mathcal{L}^{\lambda}f(t,0)&=\frac{M}{\Gamma(\lambda)}\int_0^{\infty}y^{\lambda-1}\int_0^t B\left(\frac{y}{(t-t')^{1/4}}\right)\frac{\mathcal{I}_{\frac{-\lambda - 3}{4}}f(t')}{(t-t')^{1/4}}\;dt'dy\\
	&=\frac{M}{\Gamma(\lambda)}\int_0^t (t-t')^{\frac{\lambda + 3}{4}-1}\mathcal{I}_{\frac{-\lambda - 3}{4}}f(t') \int_0^{\infty} y^{\lambda-1}B(y) \; dy dt'\\
	&=\frac{M}{\Gamma(\lambda)}\Gamma\left(\frac{\lambda}{4} + \frac34\right)f(t)\frac{\Gamma(\lambda)\Gamma\left(\frac14-\frac{\lambda}{4}\right)}{8\pi}\left(e^{-i\frac{\pi}{8}(1+3\lambda)}+e^{-i\frac{\pi}{8}(1-5\lambda)}\right)\\
		&=\frac{M}{8\ }f(t)\left(\frac{e^{-i\frac{\pi}{8}(1+3\lambda)}+e^{-i\frac{\pi}{8}(1-5\lambda)}}{\sin (\frac{1-\lambda}{4}\pi)}\right),
	\end{split}\]
where in the last equality we used the fact that $$\Gamma(z)\Gamma(1-z)=\frac{\pi}{\sin\ \pi z}.$$
Thus, the proof is complete.	\end{proof}

\subsection{Construction of the solution}\label{section4}
Let us describe how we can construct the solution for the linear fourth order Schr\"odinger equation
\begin{equation}\label{eq:lin.fourth}
i\partial_t u -\partial_x^4u =0.
\end{equation}
\subsubsection{Linear version}
First, we define the unitary group associated to \eqref{eq:lin.fourth} as
\begin{equation*}
e^{it\partial_x^4}\phi(x)=\frac{1}{2\pi}\int_{\R} e^{ix\xi}e^{-it\xi^4}\hat{\phi}(\xi)d\xi,
\end{equation*}
which allows 
\begin{equation}\label{linear}
\begin{cases}
(i\partial_t-\partial_x^4)e^{it\partial_x^4}\phi(x) =0,& (t,x)\in\mathbb{R}\times\mathbb{R},\\
e^{it\partial_x^4}\phi(x)\big|_{t=0}=\phi(x),& x\in\mathbb{R}.
\end{cases}
\end{equation}
Recall $\mathcal{L}^{\lambda}$ in \eqref{classe22} for the right half-line problem. Let
\begin{align*}
u(t,x)= \mathcal{L}^{\lambda_1}\gamma_1(t,x)+\mathcal{L}^{\lambda_2}\gamma_2(t,x),
\end{align*}
and
\begin{align*}
\partial_xu(t,x)= \mathcal{L}^{\lambda_1-1}\mathcal{I}_{-\frac14}\gamma_1(t,x)+\mathcal{L}^{\lambda_2-1}\mathcal{I}_{-\frac14}\gamma_2(t,x),
\end{align*}
where $\gamma_j$ ($j=1,2$) will be chosen later in terms of the given boundary data $f$ and $g$. 

Let $a_j$ and $b_j$ be constants depending on $\lambda_j$, $j=1,2$, given by
\begin{equation}\label{eq:entries}
a_j = \frac{M}{8\ }\left(\frac{e^{-i\frac{\pi}{8}(1+3\lambda_{j})}+e^{-i\frac{\pi}{8}(1-5\lambda_j)}}{\sin(\frac{1-\lambda_{j}}{4}\pi)}\right) \quad \text{and}\quad b_j =\frac{M}{8\ }\left(\frac{e^{-i\frac{\pi}{8}(-2+3\lambda)}+e^{-i\frac{\pi}{8}(6-5\lambda)}}{\sin(\frac{2-\lambda_j}{4}\pi)}\right).
\end{equation}
By Lemmas \ref{holmer1} and \ref{trace1}, we get
\begin{equation}\label{eq:f0}
f(t)= u(t,0) = a_1\gamma_1(t)+ a_2\gamma_2(t)
\end{equation}
and
\begin{equation}\label{eq:g0}
g(t)=\partial_x u(t,0) =b_1\mathcal{I}_{-\frac14}\gamma_1(t)+b_2\mathcal{I}_{-\frac14}\gamma_2(t).
\end{equation}
Thanks to \eqref{eq:f0} and \eqref{eq:g0}, we can write a matrix in the following form 
\[ \left[\begin{array}{c}
f(t) \\
\mathcal{I}_{\frac14}g(t)  \end{array} \right]=
A
 \left[\begin{array}{c}
\gamma_1(t)\\
\gamma_2(t)   \end{array} \right], \]
where
\[A(\lambda_1,\lambda_2)=
\left[\begin{array}{cc}
a_1 & a_2 \\
	b_1 & b_2 \end{array} \right].\]
Choosing appropriate $\lambda_j$, $j=1,2$, such that $A$ is invertible, we have that $u$ solves
\begin{equation}\label{forcante00}
\begin{cases}
(i\partial_t-\partial_x^4)u(t,x)= iM\frac{x_{-}^{\lambda_1-1}}{\Gamma(\lambda_1)}\mathcal{I}_{-\frac{3}{4}-\frac{\lambda}{4}}\gamma_1(t)+iM\frac{x_{-}^{\lambda_2-1}}{\Gamma(\lambda)}\mathcal{I}_{-\frac{3}{4}-\frac{\lambda}{4}}\gamma_2(t), & (t,x)\in\mathbb{R}^+\times\mathbb{R}, \\
u(0,x) = 0,& x\in\mathbb{R}, \\
u(t,0) = f(t), \; \px u(t,0) = g(t) ,& t\in\mathbb{R}^+.
\end{cases}
\end{equation}
By restriction of the  function $u$ on the set $\R^+\times\R^+$,  we can construct a solution for the linear fourth order dispersive equation \eqref{eq:lin.fourth} posed on the right half-line.

\subsubsection{Nonlinear version}\label{section4-1}
Now, we define the classical Duhamel inhomogeneous solution operator $\mathcal{D}$ by
\begin{equation*}
\mathcal{D}w(t,x)=-i\int_0^te^{i(t-t')\partial_x^4}w(t',x)dt'.
\end{equation*}
It follows that
\[\left \{
\begin{array}{l}
(i\partial_t-\partial_x^4)\mathcal{D}w(t,x) =w(t,x),\ (t,x)\in\mathbb{R}\times\mathbb{R},\\
\mathcal{D}w(x,0) =0,\ x\in\mathbb{R}.
\end{array}
\right.\]
Let 
\[u(t,x)= \mathcal{L}^{\lambda_1}\gamma_1(t,x)+\mathcal{L}^{\lambda_2}\gamma_2(t,x)+ e^{it\px^4}\phi(x) + \mathcal{D}w.\]
Similarly as it was done in Subsection \ref{section4}, taking  $\gamma_1$ and $\gamma_2$ appropriate depending of $f$, $g$, $e^{it\px^4}\phi(x)$ and $\mathcal{D}w$, we see that $u$ solves
\begin{equation}\label{sys4NLS}
\begin{cases}
(i\partial_t-\partial_x^4)u(t,x)= w(t,x), & (t,x)\in \R^+\times\R^+\\
u(0,x) = \phi(x), & x\in \R^+\\
u(t,0) = f(t), \; \px u(t,0) = g(t),  & t\in \R^+.
\end{cases}
\end{equation}
This discussion about the structure of the system \eqref{sys4NLS}  can be found in section \ref{sec:main proof 1} and, in this moment, will be omitted.

\section{Energy estimates}\label{sec:energy}
The main purpose of this section is to prove the energy estimate of the solutions of the fourth order nonlinear Schr\"odinger equation in the Bourgain spaces $X^{s,b}$.

\begin{lemma}\label{grupo}
	Let $s\in\mathbb{R}$ and $b \in \R$. If $\phi\in H^s(\mathbb{R})$, then the following estimates holds
	\begin{equation}\label{space}
	\|\psi(t)e^{it\partial_x^4}\phi(x)\|_{C_t\big(\mathbb{R};\,H_x^s(\mathbb{R})\big)}\lesssim_{\psi} \|\phi\|_{H^s(\mathbb{R})};
	\end{equation}
	\begin{equation}\label{derivative}
	\|\psi(t) \partial_x^{j}e^{it\partial_x^4}\phi(x)\|_{C_x\left(\mathbb{R};H_t^{\frac{2s+3-2j}{8}}(\mathbb{R})\right)}\lesssim_{\psi, s, j} \|\phi\|_{H^s(\mathbb{R})}, \quad j\in\{0,1\};
	\end{equation}
	\begin{equation}\label{bourgain}
	\|\psi(t)e^{it\partial_x^4}\phi(x)\|_{X^{s,b}}\lesssim_{\psi, b}  \|\phi\|_{H^s(\mathbb{R})}.
	\end{equation}
Estimates \eqref{space}, \eqref{derivative} and \eqref{bourgain} are so called space traces, derivative time traces and Bourgain spaces estimates, respectively.
\end{lemma}
\begin{proof}
The proofs of \eqref{space} and \eqref{bourgain} are standard and the proof of \eqref{derivative} follows from the smoothness of $\psi$ and the local smoothing estimate \eqref{eq:lsm}, thus we will omit the details.
\end{proof}

\begin{lemma}\label{duhamel}
Let $0 < b < \frac12 $ and $j=0,1$, we have the following inequalities
\begin{equation}\label{space1}
\|\psi(t)\mathcal{D}w(t,x)\|_{C\big(\mathbb{R}_t;\,H^s(\mathbb{R}_x)\big)}\lesssim \|w\|_{X^{s,-b}},
\end{equation}
for $s\in \mathbb{R}$; 
\begin{equation}\label{derivative1}
		\|\psi(t) \partial_x^j\mathcal{D}w(t,x)\|_{C\left(\mathbb{R}_x;H^{\frac{2s+3-2j}{8}}(\mathbb{R}_t)\right)}\lesssim	\|w\|_{X^{s,-b}},
\end{equation}
for $-\frac32+j<s<\frac12+j$;
\begin{equation}\label{bourgain1}
\|\psi(t) \partial_x^j\mathcal{D}w(t,x)\|_{X^{s,b}}\lesssim	\|w\|_{X^{s,-b}},
\end{equation}
for $s\in \mathbb{R}$. 

Estimates \eqref{space1}, \eqref{derivative1} and \eqref{bourgain1} are so called space traces, derivative time traces and Bourgain spaces estimates, respectively.
\end{lemma}
\begin{proof}
The idea to prove this lemma follows a variation of the proof due Kenig \textit{et al.} in \cite{KPV1991}. Here, we will give the sketch of the proof for sake of completeness.

\vspace{0.2cm}

\noindent\textbf{Estimate \eqref{space1}}:

By using $2\chi_{(0,t)}(t')=\text{sgn}t'+\sgn(t-t')$, $\widehat{\text{sgn}}(\tau)=\text{p.v.} \frac{2}{i\tau}$ and $e^{i\tau\xi^4}\hat{f}(\tau)=\hat{f}(\tau+\xi^4)$ we have
\begin{equation}\label{eq:duhamel fourier(a)}
\psi(t)\mathcal{D}w(t,x) = c\int e^{ix\xi}e^{-it\xi^4}\psi(t) \int \wt{w}(\tau',\xi) \frac{e^{it(\tau'+\xi^4) }-1}{(\tau' + \xi^4)} \; d\tau'd\xi.
\end{equation}
We denote by $w = w_1 + w_2$, where 
\[\wt{w}_1(\tau,\xi) = \eta_0(\tau+\xi^4)\wt{w}(\tau,\xi),\] and 
\[\wt{w}_2(\tau,\xi) = (1-\eta_0(\tau+\xi^4))\wt{w}(\tau,\xi).\]
Here,  $\eta_0: \R \to [0,1]$ is a smooth bump function supported in $ [-2,2]$ and equal to $1$ in $[-1,1]$.
For $w_1$, we use the  Taylor expansion of $e^x$ at $x =0$. Then, we can rewrite \eqref{eq:duhamel fourier(a)} for $w_1$ as
\[
\begin{aligned}
\psi(t)\mathcal{D}w(t,x) &= c\int e^{ix\xi}e^{-it\xi^4}\psi(t) \int \wt{w}_1(\tau',\xi) \frac{e^{it(\tau'+\xi^4) }- 1}{(\tau' + \xi^4)} \; d\tau'd\xi\\
&=c\sum_{k=1}^{\infty}\frac{i^{k-1}}{k!}\psi^k(t) \int e^{ix\xi}e^{-it\xi^4}\wh{F}_1^k(\xi)\; d\xi\\
&=c\sum_{k=1}^{\infty}\frac{i^{k-1}}{k!}\psi^k(t)e^{it\partial_x^4}F_1^k(x),
\end{aligned}
\]
where $\psi^k(t) = t^k\psi(t)$ and 
\begin{equation}\label{eq:linear estimate0}
\wh{F}_1^k(\xi) = \int \wt{w}_1(\tau,\xi) (\tau+\xi^4)^{k-1} \; d\tau.
\end{equation}
Since
\begin{equation}\label{eq:linear estimate}
\norm{F_1^k}_{H^s} = \left(\int \bra{\xi}^{2s} \left| \int \wt{w}_1(\tau,\xi) (\tau+\xi^4)^{k-1} \; d\tau \right|^2 \; d\xi \right)^{\frac12} \lesssim \norm{w}_{X^{s,-b}},
\end{equation}
we have from \eqref{space} that
\[
\norm{\psi(t) \mathcal{D}w(t,x)}_{C_tH^s} \lesssim \sum_{k=1}^{\infty}\frac{1}{k!}\norm{F_1^k}_{H_x^s} \lesssim \norm{w}_{X^{s,-b}}.
\]
For $w_2$, a direct calculation gives
\begin{equation}\label{eq:duhamel fourier}
\ft[\psi\mathcal{D}w](\tau,\xi) = c\int \wt{w}_2(\tau',\xi) \frac{\wh{\psi}(\tau-\tau') - \wh{\psi}(\tau + \xi^4)}{(\tau' + \xi^4)} \; d\tau'.
\end{equation}
Since $\norm{\psi \mathcal{D}w}_{C_tH^s} \lesssim \norm{\bra{\xi}^s\ft[\psi\mathcal{D}w](\tau,\xi)}_{L_{\xi}^2L_{\tau}^1}$, it suffices to bound the following term
\begin{equation}\label{eq:a.1}
\left(\int \bra{\xi}^{2s}\left|\int |\wt{w}_2(\tau',\xi)| \int \frac{|\wh{\psi}(\tau-\tau') - \wh{\psi}(\tau + \xi^4)|}{|\tau' + \xi^4|} \; d\tau d\tau' \right|^2 \; d\xi \right)^{\frac12},
\end{equation}
due to \eqref{eq:duhamel fourier}. We use the $L^1$ integrability of $\wh{\psi}$, to bound \eqref{eq:a.1} by
\[
c \left(\int \bra{\xi}^{2s}\left|\int_{|\tau'+\xi^4| > 1} \frac{|\wt{w}_2(\tau',\xi)|}{|\tau' + \xi^4|} d\tau' \right|^2 \; d\xi \right)^{\frac12} \lesssim \norm{w}_{X^{s,-b}}.
\]

\vspace{0.2cm}

\noindent\textbf{Estimate \eqref{derivative1}}:
 
We only  consider the case $j=0$, since the estimate for $j=1$ is a direct consequence of the case $j=0$. Initially, take $\theta(\tau)\in C^{\infty}(\mathbb{R})$ such that $\theta(\tau)=1$ for $|\tau|<\frac{1}{2}$ and supp $ \theta \subset[-\frac{2}{3},\frac{2}{3}]$. A standard calculation gives 
 \begin{eqnarray*}
 	& &\mathcal{F}_x\left( \psi(t)\int_0^te^{(t-t')\partial_x^4}w(x,t')\right)(\xi)=c\psi(t)\int_{\tau}\frac{e^{it\tau}-e^{-it\xi^4}}{\tau+\xi^4}\wt{w}(\xi,\tau)d\tau
 	\\
 	& &\quad=c \psi(t)e^{it\xi^4}\int_{\tau}\frac{e^{-it(\tau+\xi^4)}-1}{\tau+\xi^4}\theta(\tau+\xi^4)\wt{w}(\xi,\tau)d\tau+c\psi(t)\int_{\tau}e^{it\tau}\frac{1-\theta(\tau+\xi^4)}{\tau+\xi^4}\wt{w}(\xi,\tau)d\tau\\
 	& &\quad \quad -c\psi(t)e^{it\xi^4}\int_{\tau}\frac{1-\theta(\tau+\xi^4)}{\tau+\xi^4}\wt{w}(\xi,\tau)d\tau:=\mathcal{F}_xw_1+\mathcal{F}_xw_2-\mathcal{F}_xw_3.
 \end{eqnarray*}
 By the power series expansion for $e^{-it(\tau+\xi^4)}$, we have $$w_1(x,t)=\displaystyle\sum_{k=1}^{\infty}\frac{\psi_k(t)}{k!}e^{it\partial_x^4}\phi_k(x).$$ 
 Here, $\psi_k(t)=i^kt^k\theta(t)$ and $$\hat{\phi}_k(\xi)=\int_{\tau}(\tau+\xi^4)^{k-1}\theta(\tau+\xi^4)\wt{w}(\xi,\tau)d\tau.$$ By using  \eqref{derivative}, it suffices to show that $\|\phi_k\|_{H^s(\mathbb{R})}\leq c\|u\|_{X^{s,-b}}$, for $b<\frac12$. Using the definition of $\phi_k$ and the Cauchy-Schwarz inequality, follows  that
 \begin{eqnarray*}
 	\|\phi_k\|_{H^s(\mathbb{R}_x)}^2&=&c\int_{\xi}\langle \xi \rangle^{2s}\left( \int_{\{\tau:|\tau+\xi^4|\leq\frac{2}{3}\}}\sum_{k=1}^{\infty}(\tau+\xi^4)^{k-1}\theta(\tau+\xi^4)\wt{u}(\xi,\tau)\right)^2d\xi\\
 	&\leq&c\int_{\xi}\langle \xi \rangle^{2s}\int_{\tau}\langle \tau+\xi^4\rangle^{2c}|\wt{u}(\xi,\tau)|^2d\tau d\xi.
 \end{eqnarray*}
 This completes the estimate of $w_1$.  Now we treat $w_2$. By using the change of variable $\eta=\xi^4$ and Cauchy-Schwarz inequality we obtain
 \begin{equation*}
	\|w_2\|_{C\big(\mathbb{R}_x;\,H^{\frac{2s+3}{8}}(\mathbb{R}_t)\big)}^2\leq c \int_{\tau}\langle \tau\rangle^{\frac{2s+3}{4}}G(\tau)\int_{\xi}\langle \tau+\xi^4\rangle^{-2b}\langle\xi\rangle^{2s}|\wt{w}_2(\xi,\tau)|^2d\xi d\tau,
 \end{equation*}
 where $G(\tau)= c\int_{\eta}\langle \tau+\eta\rangle^{-2+2b}|\eta|^{-\frac{3}{4}}\langle\eta\rangle^{-s/2} d\eta.$ To conclude the estimate of $w_2$,
we need to prove the following estimate
\begin{equation}\label{obj}
G(\tau)\leq c \langle\tau \rangle^{-\frac{2s+3}{4}}.
\end{equation}
We split it in two cases. In the first case, we consider $|\tau|<1$. For this, we use $\langle \tau+\eta\rangle \sim \langle \eta\rangle$ to get
\begin{equation*}
G(\tau)\leq c \int \langle \eta \rangle^{-2+2b-\frac{s}{2}}|\eta|^{-3/4}d\eta.
\end{equation*}
The above integral is bounded in the case $s>-\frac72+4b$, since $-b>-\frac12$. Also, this estimate is valid for $s>-\frac32$.

 Now, the second case $|\tau|\geq 1$ can be estimated by separating the integral in three regions  
 $|\eta|\leq 1,\ 2|\eta|\leq |\tau|,\ |\tau|\leq 2|\eta|$ and using that $-\frac{3}{2}<s\leq \frac{1}{2}$, so \eqref{obj} follows. 
 
 
Finally, to bound $w_3$, let us rewrite $w_3$ like $w_3=\psi(t)e^{it\partial_x^4}\phi(x)$, where $$\hat{\phi}(\xi)=\int\frac{1-\theta(\tau+\xi^4)}{\tau+\xi^4}\wt{w}(\xi,\tau)d\tau.$$ Thanks to \eqref{derivative} and Cauchy-Schwarz inequality, we obtain
 \begin{eqnarray*}
	\|w_3\|^2_{C\left(\mathbb{R}_x;H^{\frac{2s+3}{8}}(\mathbb{R}_t)\right)}&{=}&{c\|\psi(t)e^{it\partial_x^4}\phi(x)\|_{C\left(\mathbb{R}_x;H^{\frac{2s+3}{8}}(\mathbb{R}_t)\right)}^2\leq c\|\phi\|_{H^s(\mathbb{R})}^2}\\
 	&{\leq}&{c\int_{\xi}\langle \xi \rangle^{2s}\left(\int_{\tau}|\wt{w}(\xi,\tau)|^2\langle\tau+\xi^4\rangle^{{-2b}}d\tau \int \frac{d\tau}{\langle\tau+\xi^4\rangle^{2{-2b}}}\right) d\xi.}
 \end{eqnarray*}
 Since $b<\frac{1}{2}$, we have $\int \frac{1}{\langle\tau+\xi^4\rangle^{2-2b}}d\tau\leq c$. By using \eqref{obj}, estimate \eqref{derivative1} for $w_3$ follows and, consequently, \eqref{derivative1}  holds true for $w=w_1+w_2+w_3$.

\vspace{0.2cm}

\noindent\textbf{Estimate \eqref{bourgain1}}:

Finally,  again we split $w=w_1+w_2$, similarly as it was done in the proof of \eqref{space1}. For $w_1$, estimates \eqref{bourgain} and \eqref{eq:linear estimate} yield that
\[\norm{\psi(t) \mathcal{D}w_1(t,x)}_{X^{s,b}} \lesssim \sum_{k=1}^{\infty}\frac{1}{k!}\norm{F_1^k}_{H_x^s} \lesssim \norm{w}_{X^{s,-b}},\]
where $F_1^k$ is defined as in \eqref{eq:linear estimate0}.
 
For $w_2$, note that
\[
\begin{aligned}
\psi\partial_x^j\mathcal{D}w(t,x) =& c\int e^{ix\xi}e^{-it\xi^4}(i\xi)^j\psi(t) \int \frac{\wt{w}(\tau',\xi)}{(\tau' + \xi^4)} \left(e^{it(\tau'+\xi^4) }- 1\right) \; d\tau'd\xi\\
=&c\int e^{ix\xi}e^{-it\xi^4}(i\xi)^j\psi(t) \int \frac{\wt{w}(\tau',\xi)}{(\tau' + \xi^4)} e^{it(\tau'+\xi^4) } \; d\tau'd\xi\\
& -c\int e^{ix\xi}e^{-it\xi^4}(i\xi)^j\psi(t) \int \frac{\wt{w}(\tau',\xi)}{(\tau' + \xi^4)}  \; d\tau'd\xi\\
=& I -II.
\end{aligned}
\]
Let \[\wh{W}(\xi) = \int \frac{\wt{w}_2(\tau,\xi)}{(\tau+\xi^4)} \; d\tau.\]
Therefore, we use \eqref{bourgain} in $II$ to obtain
\[\norm{\psi  e^{it\partial_x^4}W}_{X^{s,b}} \lesssim \norm{W}_{H^s} \lesssim \norm{w}_{X^{s,-b}},\]
 for $b < \frac12$.

Now, it remains to show the following estimate
\begin{equation}\label{eq:c.4}
\Big(\int_{|\xi| > 1}|\xi|^{2s}\int\bra{\tau+\xi^4}^{2b} \Big| \int \frac{\wt{w}_2(\tau',\xi)}{i(\tau'+\xi^4)}\wh{\psi}(\tau - \tau') \; d\tau'\Big|^2 \; d\tau d\xi\Big)^{\frac12}\lesssim \norm{w}_{X^{s,-b}}.
\end{equation}
This follows by using the same argument as we did to prove \eqref{derivative1}. In fact, the proof of \eqref{eq:c.4}  is easier than proof of \eqref{derivative1}, since $L^2$ integral with respect to $\xi$ is negligible and hence it is enough to consider the relation between $\tau + \xi^4$ and $\tau'+\xi^4$. Thus, as consequence, we have
\[\norm{\psi \mathcal{D}w}_{X^{s,b}} \lesssim \norm{w}_{X^{s,-b}}.\]
Therefore, Lemma \ref{duhamel} is proved.
\end{proof}

\begin{lemma}\label{edbf}
	Let $s\in\mathbb{R}$. 
\begin{itemize}
\item[(a)] For $\frac{2s-7}{2} <\lambda< \frac{1+2s}{2}$ and $\lambda< \frac12$ the following inequality holds
		\begin{equation}\label{space2}
		\|\psi(t)\mathcal{L}^{\lambda}f(t,x)\|_{C\big(\mathbb{R}_t;\,H^s(\mathbb{R}_x)\big)}\leq c \|f\|_{H_0^\frac{2s+3}{8}(\mathbb{R}^+)};
		\end{equation}
\item[(b)] For $-4+j<\lambda<1+j$, $j=0,1$, we have 
		\begin{equation}\label{eq:(b)0}
		\|\psi(t)\partial_x^j\mathcal{L}^{\lambda}f(t,x)\|_{C\big(\mathbb{R}_x;\,H_0^{\frac{2s+3-2j}{8}}(\mathbb{R}_t^+)\big)}\leq c \|f\|_{H_0^\frac{2s+3}{8}(\mathbb{R}^+)};
		\end{equation}
\item[(c)] If $ s <4-4b $, $b < \frac12$, $-5<\lambda<\frac12$ and $s+4b-2<\lambda<s+\frac12$ yields that
\begin{equation}\label{bourgain2}
\|\psi(t)\mathcal{L}^{\lambda}f(t,x)\|_{X^{s,b}}\leq c \|f\|_{H_0^\frac{2s+3}{8}(\mathbb{R}^+)}.
\end{equation}
\end{itemize}
Estimates \eqref{space2}, \eqref{eq:(b)0} and \eqref{bourgain2} are so called space traces, derivative time traces and Bourgain spaces estimates, respectively.
\end{lemma}

\begin{proof}
Let us first to prove \eqref{space2}. By density, we may assume that $f\in C_{0,c}^{\infty}(\mathbb{R}^{+})$. Moreover, from definition of $\mathcal{L}^{\lambda}$, it suffices to consider $\mathcal{L}^{\lambda}f(t,x)$ (removing $\psi$) for $\supp f \subset [0,1]$, thanks to Lemma \ref{cut}.

From \eqref{transformada}, \eqref{eq:BFO} and \eqref{eq:BFOC}, we see that 
\[\mathcal{F}_x(\mathcal{L}^{\lambda}f)(t,\xi)=Me^{-\frac{i\pi\lambda}{2}}(\xi-i0)^{-\lambda}\int_0^te^{-i(t-t')\xi^4}\mathcal{I}_{-\frac{\lambda}{4}-\frac{3}{4}}f(t') \;dt'.\] 
By using the following change of variable $\eta=\xi^4,$  \eqref{transformada2} and the definition of the Fourier transform we have that
	\begin{eqnarray*}
		\|\mathcal{L}^{\lambda}f(t,\cdot)\|_{H^s(\mathbb{R})}^2&\leq& c \int_{\eta}|\eta|^{-\frac{\lambda}{2}-\frac{3}{4}}\langle\eta\rangle^{\frac{s}{2}}\left|\int_0^te^{-i(t-t')\eta}\mathcal{I}_{-\frac{\lambda}{4}-\frac{3}{4}}f(t')dt'\right|^2d\eta\\
		&=&c\int_{\eta}|\eta|^{-\frac{\lambda}{2}-\frac{3}{4}}\langle\eta\rangle^{\frac{s}{2}}\left|\big(\chi_{(-\infty,t)}\mathcal{I}_{-\frac{\lambda}{4}-\frac{3}{4}}f\big)^{\widehat{}}(\eta)\right|^2d\eta,
	\end{eqnarray*}
for a fixed $t$.  Note that, by Lemma \ref{sobolev0}, we can replace  $|\eta|^{-\frac{\lambda}{2}-\frac{3}{4}}$ by $\bra{\eta}^{-\frac{\lambda}{2}-\frac{3}{4}}$,  since $$-1 < -\frac{\lambda}{2}-\frac{3}{4}\Leftrightarrow \lambda < \frac12.$$  Moreover, Lemma \ref{sobolevh0} (under the condition $-1 < -\frac{\lambda}{2}-\frac{3}{4} + \frac{s}{2} < 1$ for removing $\chi_{(-\infty,t)})$ and Lemma \ref{lio} (under the condition $-5 < \lambda$) yield that
\[
\begin{aligned}
\int_{\eta}|\eta|^{-\frac{\lambda}{2}-\frac{3}{4}}\langle\eta\rangle^{\frac{s}{2}}\left|(\chi_{(-\infty,t)}\mathcal{I}_{-\frac{\lambda}{4}-\frac{3}{4}}f)^{\widehat{}}(\eta)\right|^2d\eta &\leq c\int_{\eta}\langle\eta\rangle^{\frac{s}{2}-\frac{\lambda}{2}-\frac{3}{4}}\left|(\chi_{(-\infty,t)}\mathcal{I}_{-\frac{\lambda}{4}-\frac{3}{4}}f)^{\widehat{}}(\eta)\right|^2d\eta\\
&\leq c\norm{\mathcal{I}_{-\frac{\lambda}{4}-\frac{3}{4}}f}_{H^{\frac{s}{4}-\frac{\lambda}{4}-\frac{3}{8}}}^2 \leq c \|f\|_{H_0^{\frac{2s+3}{8}}}^2,	
\end{aligned}
\]
which proves \eqref{space2} thanks to the definition of $H_0^s(\R^+)$-norm. 

Now we prove \eqref{eq:(b)0}. A direct calculation gives
\[\px^j\mathcal{L}^{\lambda}f = \mathcal{L}^{\lambda - j}(\mathcal{I}_{-\frac{j}{4}}f).\]
With the previous equality in hand and Lemma \ref{lio}, it suffices to show \eqref{eq:(b)0} for $j=0$. Lemma \ref{cut} ensures us to ignore the cut-off function $\psi$. The change of variable $t\rightarrow  t-t'$ gives
$$(I-\partial_t^2)^{\frac{2s+3}{16}}\left(\frac{x_{-}^{\lambda-1}}{\Gamma(\lambda)}*\int_{-\infty}^te^{i(t-t')\partial_x^4}\delta(x)h(t')dt'\right)=\left(\frac{x_{-}^{\lambda-1}}{\Gamma(\lambda)}*\int_{-\infty}^te^{i(t-t')\partial_x^4}\delta(x)(I-\partial_{t'}^2)^{\frac{2s+3}{16}}h(t')dt'\right). $$
So, we just need to prove that
\begin{equation}\label{eq:(b)}
\left\|\int_{\xi}e^{ix\xi}(\xi-i0)^{-\lambda}\int_{-\infty}^te^{-i(t-t')\xi^4}(\mathcal{I}_{-\frac{\lambda}{4}-\frac{3}{4}}f)(t')dt'd\xi\right\|_{L_x^{\infty}L_t^2(\mathbb{R})} \leq c \|f\|_{L_t^2(\mathbb{R}^+)},
\end{equation}
thanks to $\pt^{\sigma}(\mathcal{I}_{\alpha}f) = \mathcal{I}_{\alpha}(\pt^{\sigma}f)$. We use $\chi_{(-\infty,t)}=\frac{1}{2}\mbox{sgn}(t-t')+\frac{1}{2}$ to obtain
\[\begin{aligned}
\int_{\xi}e^{ix\xi}(\xi-i0)^{-\lambda}&\int_{-\infty}^te^{-i(t-t')\xi^4}(\mathcal{I}_{-\frac{\lambda}{4}-\frac{3}{4}}f)(t') \; dt'd\xi\\
=&~{}\frac12\int_{\xi}e^{ix\xi}(\xi-i0)^{-\lambda}\int_{-\infty}^{\infty}\mbox{sgn}(t-t')e^{-i(t-t')\xi^4}(\mathcal{I}_{-\frac{\lambda}{4}-\frac{3}{4}}f)(t') \; dt'd\xi\\
&+\frac12\int_{\xi}e^{ix\xi}(\xi-i0)^{-\lambda}\int_{-\infty}^{\infty}e^{-i(t-t')\xi^4}(\mathcal{I}_{-\frac{\lambda}{4}-\frac{3}{4}}f)(t') \; dt'd\xi\\
:=& I(t,x)+II(t,x).
\end{aligned}\]
We will treat $I(t,x):=I$ and $II(t,x):=II$ separately. To estimate $I$, we can rewrite it as
\[I(t,x)=\frac12\int_{\xi}e^{ix\xi}(\xi-i0)^{-\lambda} \left((e^{-i\cdot\xi^4}\mbox{sgn}(\cdot))*\mathcal{I}_{-\frac{\lambda}{4}-\frac{3}{4}}f\right)(t) \; d\xi.\]
A direct calculation gives
$$\mathcal{F}_t\left((e^{-i\cdot\xi^4}\mbox{sgn}(\cdot))*\mathcal{I}_{-\frac{\lambda}{4}-\frac{3}{4}}f\right)(\tau)=\frac{(\tau-i0)^{\frac{3+\lambda}{4}}\hat{f}(\tau)}{i(\tau+\xi^4)}.$$ 
Fubini's theorem and dominated converge theorem imply that
\begin{equation*}
I(t,x)=\int_{\tau}e^{it\tau}\lim_{\epsilon\rightarrow 0}\int_{|\tau+\xi^4|>\epsilon}\frac{e^{ix\xi}(\tau-i0)^{\frac{\lambda+3}{4}}(\xi-i0)^{-\lambda}}{i(\tau+\xi^4)}\wh{f}(\tau) \; d\xi d\tau.
\end{equation*}
Thus, once we show that the function 
$$g(\tau):=\lim_{\epsilon\rightarrow 0}\int_{|\tau+\xi^4|>\epsilon}\frac{e^{ix\xi}(\tau-i0)^{\frac{\lambda+3}{4}}(\xi-i0)^{-\lambda}}{(\tau+\xi^4)} \; d\xi$$
is bounded independently of $\tau$ variable, the Plancherel's theorem enables us to obtain \eqref{eq:(b)}. The change of variable $\xi \mapsto |\tau|^{\frac{1}{4}}\xi$  and the fact that
\[(|\tau|^{\frac{1}{4}}\xi-i0)^{-\lambda}=|\tau|^{-\frac{\lambda}{4}}(\xi_{+}^{-\lambda}+e^{i\pi\lambda}\xi_{-}^{-\lambda})\]
gives
\begin{eqnarray*}
	g(\tau)&=&\chi_{\{\tau>0\}}\int_{\xi}e^{ix|\tau|^{\frac{1}{4}}\xi} \; \frac{\xi_{+}^{-\lambda}+e^{i\pi\lambda}\xi_{-}^{-\lambda}}{1+\xi^4} \; d\xi - e^{-\frac{i\pi(\lambda + 3)}{4}}\chi_{\{\tau<0\}}\int_{\xi}e^{ix|\tau|^{\frac{1}{4}}\xi} \; \frac{\xi_{+}^{-\lambda}+e^{i\pi\lambda}\xi_{-}^{-\lambda}}{1-\xi^4} \; d\xi\\
	&:=&g_1 - e^{-\frac{i\pi(\lambda + 3)}{4}}g_2.
\end{eqnarray*}
We only consider $g_2$, since $g_1$ is uniformly bounded in $\tau$ for $-3 < \lambda < 1$. Let us define the following cut-off function $\zeta\in C^{\infty}(\mathbb{R})$ such that 
\begin{equation*}
\zeta:=\begin{cases}
1, &\text{in } [\frac{3}{4},\frac{4}{3}]\\
0, &\text{outside } (\frac{1}{2},\frac{3}{2}).
\end{cases}
\end{equation*}
Then, we obtain
\begin{eqnarray*}
	g_2&=&\chi_{\{\tau<0\}}\int_{\xi}e^{ix|\tau|^{\frac{1}{4}}\xi} \zeta(\xi) \frac{\xi_{+}^{-\lambda}}{1-\xi^4} \; d\xi+ \chi_{\{\tau<0\}}\int_{\xi}e^{ix|\tau|^{\frac{1}{4}}\xi} (1-\zeta(\xi)) \frac{\xi_+^{-\lambda}+e^{i\pi\lambda}\xi_{-}^{-\lambda}}{1-\xi^4} \; d\xi\\
	&=&g_{21}+g_{22}.
\end{eqnarray*}
It is clear that $g_{22}$ is bounded independently of $\tau$ when $\lambda>-3$, and hence it remains to deal with $g_{21}$. Consider the following functions
\[\wh{\Theta}(\xi) = \frac{\zeta(\xi) \xi_+^{-\lambda}}{1+\xi + \xi^2 + \xi^3 } \quad \mbox{and} \quad \wh{\Psi}(\xi) = \frac{1}{i(\xi - 1)}.\]
We remark that $\wh{\Theta}$ is a Schwartz function, and hence $\Theta \in \Sch(\R)$. Moreover, we immediately know  that
\[\Psi(x) = \frac12 e^{ix} \mbox{sgn}(x),\]
since $\ft_x[\mbox{sgn}(x)](\xi) = \text{v.p.} \frac{2}{i\xi}$.
Then, $g_{21}$ can be written as
\[g_{21}(\tau) = -i\chi_{\{\tau<0\}}\int_{\xi}e^{ix|\tau|^{\frac{1}{4}}\xi}\wh{\Theta}(\xi)\wh{\Psi}(\xi) \; d\xi = -2i\pi \chi_{\{\tau<0\}}(\Theta \ast \Psi)(|\tau|^{\frac14}x),\]
which implies
\[
\begin{aligned}
|g_{21}(\tau)| &\lesssim \left|\int \Theta(y)\Psi(|\tau|^{\frac14}x - y) \; dy  \right|\lesssim \int |\Theta(y)| \; dy \lesssim_{\zeta} 1.
\end{aligned}
\]

Now, we bound $II$. By using the definition of Fourier transform and \eqref{transformada2} we have, after the change of variable $\eta=\xi^4$ and contour, that
\[\begin{aligned}
	II(t,x)=&~{}\frac12\int_{\xi}e^{ix\xi}e^{-it\xi^4}(\xi^4-i0)^{\frac{\lambda+3}{4}}\wh{f}(\xi^4)(\xi-i0)^{-\lambda} \; d\xi\\
	=&~{}\frac12\int_{0}^{+\infty}e^{it\eta}e^{-ix\eta^{\frac{1}{4}}}(\eta-i0)^{\frac{\lambda+3}{4}}(\eta^{\frac{1}{4}}-i0)^{-\lambda}\eta^{-\frac{3}{4}}\wh{f}(\eta) \;d\eta\\
	=&~{}cf(t),
\end{aligned}\]
for some $c \in \mathbb{C}$, which implies $\|II(\cdot,x)\|_{L_t^2} \lesssim  \|f\|_{L_t^2}$. Therefore, we complete the proof of \eqref{eq:(b)0}.	
	
Lastly, let us show \eqref{bourgain2}. A direct calculation ensures that
\begin{equation*}
\mathcal{F}_x(\psi(t)\mathcal{L}^{\lambda}f)(t,\xi)=Me^{-\frac{i\pi\lambda}{2}}e^{\frac{i\pi(\lambda+4)}{10}}(\xi-i0)^{-\lambda}\psi(t)e^{-it\xi^4}\int\frac{e^{it(\tau'+\xi^4)}-1}{i(\tau'+\xi^4)}(\tau'-i0)^{\frac{\lambda}{4}+\frac{3}{4}}\wh{f}(\tau')\; d\tau',
\end{equation*}
which can be divided into the following quantities
\begin{equation*}
\wh{f}_1(t,\xi)=Me^{-\frac{i\pi\lambda}{2}}e^{\frac{i\pi(\lambda+4)}{10}}(\xi-i0)^{-\lambda}\psi(t)\int\frac{e^{it\tau'}-e^{-it\xi^4}}{i(\tau'+\xi^4)}\theta(\tau'+\xi^4)(\tau'-i0)^{\frac{\lambda}{4}+\frac{3}{4}}\wh{f}(\tau')\; d\tau',
\end{equation*} 
\begin{equation*}
\wh{f}_2(t,\xi)=Me^{-\frac{i\pi\lambda}{2}}e^{\frac{i\pi(\lambda+4)}{10}}(\xi-i0)^{-\lambda}\psi(t)\int\frac{e^{it\tau'}}{i(\tau'+\xi^4)}(1-\theta(\tau'+\xi^4))(\tau'-i0)^{\frac{\lambda}{4}+\frac{3}{4}}\wh{f}(\tau')\; d\tau'
\end{equation*} 
and
\begin{equation*}
\wh{f}_3(t,\xi)=Me^{-\frac{i\pi\lambda}{2}}e^{\frac{i\pi(\lambda+4)}{10}}(\xi-i0)^{-\lambda}\psi(t)\int\frac{e^{-it\xi^4}}{i(\tau'+\xi^4)}(1-\theta(\tau'+\xi^4))(\tau'-i0)^{\frac{\lambda}{4}+\frac{3}{4}}\wh{f}(\tau')\; d\tau'.
\end{equation*} 
Here $\theta \in \mathcal{S}(\R)$ is defined by
\begin{equation}\label{theta}
\theta(\tau):=\begin{cases}
1, &\text{for } |\tau|\leq 1\\
0, &\text{for } |\tau|\geq 2.
\end{cases}
\end{equation}
 It follows that $\psi(t)\mathcal{L}^{\lambda}f= f_1 + f_2 - f_3$.  

For $f_1$, we use the same argument as it was done for $w_1$, in the proof of  inequality \eqref{bourgain1}. By the Taylor series expansion for $e^{it(\tau'+\xi^4)}$ at $it(\tau'+\xi^4) = 0$, we write 
\[
\psi(t)\mathcal{L}^{\lambda}f_1(t,x) = c\sum_{k=1}^{\infty}\frac{i^{k-1}}{k!}\psi^k(t)e^{it\partial_x^4}F_1^k(x),
\]
for some constant $c \in \mathbb{C}$, where $\psi^k(t) = t^k\psi(t)$ and 
\[
\wh{F}_1^k(\xi) = (\xi-i0)^{-\lambda}\int \theta(\tau'+\xi^4)(\tau'+\xi^4)^{k-1}\tau'^{\frac{\lambda}{4}+\frac{3}{4}}\wh{f}(\tau') \; d\tau'.
\]	
By using \eqref{transformada2}, \eqref{bourgain1} and \eqref{theta}, it is enough to show that
\begin{equation}\label{crb100}
\int_{\xi}\langle \xi\rangle^{2s}|\xi|^{-2\lambda}\left|\int_{|\tau'+\xi^4|\leq 1}|\tau'+\xi^4|^{k-1}|\tau'|^{\frac{\lambda+3}{4}}|\wh{f}(\tau')| \; d\tau' \right|^2 d\xi \lesssim \norm{f}_{H_0^{\frac{2s+3}{8}}}^2.
\end{equation}
Let us split $|\xi|$ in two regions: $|\xi| \le 1$ and $|\xi| > 1$. For the region $|\xi| \le 1$ and $|\tau'| \lesssim 1$ ($|\xi| \le 1$ and $|\tau'+\xi^4| \le 1$ imply $|\tau'| \lesssim 1$) we have that both $|\xi|^{-2\lambda}$ and $|\tau'|^{\frac{(\lambda+3)}{2}}$ are integrable,  for $-5 < \lambda < \frac12$, respectively. So, we obtain \eqref{crb100} by using Cauchy-Schwarz inequality in $\tau'$.

 Now, assume that $|\xi| > 1$, which in addition with $|\tau'+\xi^4| \le 1$ implies $|\tau'| \sim |\xi|^4 > 1$. Let $\wh{f}^{\ast}(\tau') = \bra{\tau'}^{\frac{2s+3}{8}}\wh{f}(\tau')$. Then, the change of variable $\xi^4 \mapsto \eta$, gives that the left-hand side of \eqref{crb100} is bounded by
\[
 \int_{|\xi| > 1}|\xi|^3 |\mathcal{M}\wh{f}^{\ast}(\xi^4)|^2 \; d\xi\lesssim \int_{|\eta| > 1}|\mathcal{M}\wh{f}^{\ast}(\eta)|^2 \; d\eta\lesssim \norm{f^{\ast}}_{L^2}^2 = \norm{f}_{H_0^{\frac{2s+3}{8}}}^2,
\]
where $\mathcal{M}\wh{f}^{\ast}$ is the Hardy-Littlewood maximal function of $\wh{f}^{\ast}$, and $f_1$ is controlled.

For $f_2$, from \eqref{transformada2}, the definition of inverse Fourier transform and Lemma \ref{lem:Xsb}, follows that
\begin{eqnarray*}
	\|f_2\|_{X^{s,b}}^2&\lesssim& \int\int\langle \xi\rangle^{2s}|\xi|^{-2\lambda}\langle \tau+\xi^4\rangle^{2b}\frac{(1-\theta(\tau+\xi^4))^2}{|\tau+\xi^4|^2}|\tau|^{\frac{\lambda+3}{2}}|\wh{f}(\tau)|^2 \; d\tau d\xi\\
	&\lesssim&\int  |\tau|^{\frac{\lambda+3}{2}}\left( \int\frac{\langle \xi\rangle^{2s}|\xi|^{-2\lambda} }{\langle \tau+\xi^4\rangle^{2-2b}}\; d\xi\right)|\hat{f}(\tau)|^2 \; d\tau.
\end{eqnarray*}
Thus, by the change of variable $\eta = \xi^4$ and Lemma \ref{sobolev0}, for $-5 < \lambda$ (we may assume $\supp f \subset [0,1]$, thanks to Lemma \ref{cut}), it suffices to show
\[I(\tau)=\int\frac{|\eta|^{-\frac{3}{4}-\frac{\lambda}{2}}\langle \eta\rangle^{\frac{s}{2}}}{\langle \tau+\eta\rangle^{2-2b}} \; d\eta \lesssim  \langle\tau\rangle^{\frac{2s}{4}-\frac{2\lambda}{4}-\frac{3}{4}}.\]
Here, we split $|\tau|$ in two regions: $|\tau|\leq2$ and $|\tau|>2$. When $|\tau|\leq2$, we have $\langle\tau+\eta\rangle\sim\langle\eta\rangle$. For $s<4-4b$  and $s+4b-\frac72 < \lambda<\frac{1}{2}$, we get
\begin{equation*}
I(\tau)\lesssim \int_{|\eta|\leq 1} |\eta|^{\frac{-3-2\lambda}{4}}+\int\frac{d\eta}{\langle \eta\rangle^{2-2b-\frac{s}{2}+\frac{3}{4}+\frac{\lambda}{2}}}\lesssim 1.
\end{equation*}
Now, working in the region $|\tau|>2$, we divide the integral region in $\eta$ into $|\eta|< \frac{|\tau|}{2}$ and  $|\eta| \ge \frac{|\tau|}{2}$. In the first region, for $b<\frac12$ and $\lambda < \min(\frac12, s+\frac12)$, we bound in the following way 
\begin{equation*}
 \langle\tau\rangle^{2b-2}\left(\int_{|\eta| \le 1}|\eta|^{-\frac{3}{4}-\frac{2\lambda}{4}} \; d\eta +\int_{1 < |\eta|\leq \frac{|\tau|}{2}}|\eta|^{\frac{-3-2\lambda+2s}{4}}d\eta \right) \lesssim \langle\tau\rangle^{\frac{-3-2\lambda+2s}{4}}.
\end{equation*} 
On the other hand,  in the second region, we have that $|\tau+\eta| \geq \frac{1}{2} |\tau| > 1$. Then, for $s-2 < \lambda$ and $b<\frac{1}{2}$, holds that
\begin{equation*}
\begin{aligned}
I(\tau) &\lesssim \langle\tau\rangle^{\frac{-3-2\lambda+2s}{4}}\int \frac{d\eta}{\langle\tau+\eta\rangle^{2-2b}} \lesssim \langle\tau\rangle^{\frac{-3-2\lambda+2s}{4}}\int_{|s| > 1} \frac{ds}{|s|^{2-2b}}\lesssim \langle\tau\rangle^{\frac{-3-2\lambda+2s}{4}},
\end{aligned}
\end{equation*}
so, 
$$\norm{f_2}_{X^{s,b}}\lesssim \norm{f}_{H_0^{\frac{2s+3}{8}}},$$
this completes the estimate for $f_2$.

Finally, let us show that $f_3$ can be controlled. Similarly as for $f_1$, it suffices to show 
\begin{equation}\label{23110}
\int\langle \xi\rangle^{2s}|\xi|^{-2\lambda}\left|\int (1-\theta(\tau'+\xi^4))|\tau'+\xi^4|^{-1}|\tau'|^{\frac{\lambda+3}{4}}|\wh{f}(\tau')| \; d\tau' \right|^2 d\xi \lesssim \norm{f}_{H_0^{\frac{2s+3}{8}}}^2.
\end{equation}
Again, we split the region $|\xi|$ as follows: $|\xi| \le 1$ and $|\xi| >1$. Considering $|\xi| \le 1$, since $|\xi|^{-2\lambda}$ is integrable, for $\lambda < \frac12$, and we may ignore the integration in $\xi$. Let us work in the region $|\tau'| \le 1$. In this region $|\tau'|^{\frac{(\lambda+3)}{2}}$ is integrable, for $ \lambda>-5$, and hence we get \eqref{23110}. 

On the region $|\tau'| > 1$, since $|\tau' +\xi^4| \sim |\tau'|$ and $|\tau'|^{\frac{-s+\lambda -3}{5}}$ are $L^2$ integrable, for $\lambda < s+\frac12$, we also get \eqref{23110} by using the Cauchy-Schwarz inequality in $\tau'$. Still looking on the region $|\tau'| > 1$, since $|\tau' +\xi^4| \sim |\tau'|$ we have that the left-hand side of \eqref{23110} is bounded by
\begin{equation}
\begin{split}
c\left(\int_{|\tau'|>1} |\tau'|^{\frac{\lambda-1}{4}}|\hat{f}|^2d\tau'\right)^2&\sim \left(\int_{|\tau'|>1} \frac{|\tau'|^{\frac{\lambda-1}{4}}}{\langle \tau'\rangle^{\frac{2s+3}{8}}}\langle \tau'\rangle^{\frac{2s+3}{8}}|\hat{f}|^2d\tau'\right)^2\\
&\lesssim \int \frac{\langle\tau'\rangle^{\frac{\lambda-1}{2}}}{\langle\tau'\rangle^{\frac{2s+3}{8}}}d\tau' \|f\|_{H^{\frac{2s+3}{4}}}^2\lesssim \|f\|_{H_0^{\frac{2s+3}{8}}}^2,
\end{split}
\end{equation}
where we have used that $\lambda<s+\frac12$, and the result follows on $|\xi| \le 1$. On the other hand, in the region $|\xi| > 1$ and $|\tau'| \le 1$, since $|\tau' +\xi^4| \sim |\xi|^4\sim \langle \xi \rangle^4$ and $\langle \xi\rangle^{2s-2\lambda-8
 }$ are integrable for $\lambda>-\frac72+s$, we also get \eqref{23110}.

Considering the region $|\xi| > 1$ and $|\tau'| > 1$. There are two possibilities:
	\vspace{0.1cm}
\begin{itemize} \item[I.]  $|\tau'| \le \frac12|\xi|^4$;
	\vspace{0.1cm}
	\item[II.] $\frac12|\xi|^4 <|\tau'|$.
	\end{itemize}
		\vspace{0.1cm}
	 In view of the proof of \cite[Lemma 5.8 (d)]{Holmerkdv} (see also \cite{HolmerNLS}), one can replace $\frac{1-\theta(\tau'+\xi^4)}{\tau'+\xi^4}$ by $\beta(\tau'+\xi^4)$ for some $\beta \in \Sch(\R)$. Hence, the left-hand side of \eqref{23110} is dominated by
\begin{equation}\label{eq:23110}
c\int_{|\xi| > 1}|\xi|^{2s-2\lambda}\left|\int_{|\tau'|>1} |\tau'+\xi^4|^{-N}|\tau'|^{\frac{\lambda+3}{4}}|\wh{f}(\tau')| \; d\tau' \right|^2 d\xi,
\end{equation}
for $N \ge 0$. By Cauchy-Schwarz inequality and choosing $N = N(s,\lambda) \gg 1$, we have \eqref{23110} for both cases. Indeed, for the case $I$ (in this case $|\tau'+\xi^4| \sim |\xi|^4$), \eqref{eq:23110} can be controlled by
\[
c \norm{f}_{H^{\frac{2s+3}{8}}}^2\int_{|\xi| > 1}|\xi|^{2s-2\lambda-4N}  \int_{1 < |\tau'| \le \frac12|\xi|^4} |\tau'|^{\frac{-2s-3 + 2\lambda+6-4N}{4}}d\tau'd\xi \lesssim \norm{f}_{H_0^{\frac{2s+3}{8}}}^2.
\]
For case $II$ (in this case $|\tau'+\xi^4| \sim |\tau'|$), \eqref{eq:23110} is bounded by
\[c \int_{|\xi| > 1}|\xi|^{2s-2\lambda}\left|\int_{\frac12|\xi|^4 <  |\tau'|} |\tau'|^{\frac{2\lambda+6-2s-3- 4N}{4}}|\tau'|^{\frac{2s+3}{8}}|\wh{f}(\tau')| \; d\tau' \right|^2 d\xi \lesssim \norm{f}_{H_0^{\frac{2s+3}{8}}}^2,\]
thus
$$\norm{f_3}_{X^{s,b}}\lesssim \norm{f}_{H_0^{\frac{2s+3}{8}}},$$ this finishes the estimate for $f_3$. 

Remembering that $\psi(t)\mathcal{L}^{\lambda}f= f_1 + f_2 - f_3$, and using the estimates of $f_i$, $i=1,2,3$,  \eqref{bourgain2} follows and the proof is complete.
\end{proof}

To close this section, let us enunciate the trilinear estimates associated to the fourth order nonlinear Schr\"odinger equation \eqref{fourth}. The proof of this estimate can be found in \cite{Tzvetkov} (see also \cite{CaCa}), thus we will omit it.

\begin{proposition}\label{prop:bi1}
	For $s\geq 0 $, there exists $b = b(s) < 1/2$ such that we have
	\begin{equation}\label{eq:bilinear1}
	\norm{u_1u_2\overline{u}_3}_{X^{s,-b}} \le c\norm{u_1}_{X^{s,b}}\norm{u_2}_{X^{s,b}}\norm{u_3}_{X^{s,b}}.
	\end{equation}
\end{proposition}

\section{Proof of Theorem \ref{theorem1}}\label{sec:main proof 1}

Initially, we pick an extension $\tilde{u}_0\in H^s(\mathbb{R})$ of $u_0$ such that $$\|\tilde{u}_0\|_{H^s(\mathbb{R})}\leq 2\|u_0\|_{H^s(\mathbb{R}^+)}.$$ Let $b=b(s)<\frac{1}{2}$   such that the estimates given in Proposition \ref{prop:bi1} are valid. 

By using similar arguments as it was done in Subsection \ref{section4}, let
\begin{equation}\label{eq:solution}
u(t,x)= \mathcal{L}^{\lambda_1}\gamma_1(t,x)+\mathcal{L}^{\lambda_2}\gamma_2(t,x)+F(t,x),
\end{equation}
where $\gamma_i$ ($i=1,2$) will be chosen in terms of initial and boundary data $u_0$, $f$, $g$ and $F(t,x)=e^{it\partial_x^4}\tilde{u}_0 +\lambda\mathcal{D}(|u|^2u)$. 

Remember that $a_j $ and $b_j$ are defined by
\begin{equation}\label{eq:entries2}
a_j = \frac{M}{8\ }\left(\frac{e^{-i\frac{\pi}{8}(1+3\lambda_{j})}+e^{-i\frac{\pi}{8}(1-5\lambda_j)}}{\sin(\frac{1-\lambda_{j}}{4}\pi)}\right)\quad \text{and}\quad b_j =\frac{M}{8\ }\left(\frac{e^{-i\frac{\pi}{8}(-2+3\lambda_j)}+e^{-i\frac{\pi}{8}(6-5\lambda_j)}}{\sin (\frac{2-\lambda_j}{4}\pi)}\right).
\end{equation}
By Lemmas \ref{holmer1} and \ref{trace1}, we get
\begin{equation}\label{eq:f}
f(t)= u(t,0) = a_1\gamma_1(t)+a_2\gamma_2(t)+F(t,0)
\end{equation}
and
\begin{equation}\label{eq:g}
g(t)=\partial_x u(t,0) =b_1\mathcal{I}_{-\frac14}\gamma_1(t)+b_2\mathcal{I}_{-\frac14}\gamma_2(t)+ \partial_xF(t,0).
\end{equation}
Observe that putting together \eqref{eq:f} and \eqref{eq:g}, we can write a matrix in the following form 
\[ \left[\begin{array}{c}
f(t)-F(t,0) \\
\mathcal{I}_{\frac14}g(t)- \mathcal{I}_{\frac14}\partial_x F(t,0)  \end{array} \right]=
A
 \left[\begin{array}{c}
\gamma_1(t)\\
\gamma_2(t)   \end{array} \right], \]
where
\[A(\lambda_1,\lambda_2)=
\left[\begin{array}{cc}
a_1 & a_2 \\
	b_1 & b_2 \end{array} \right].\]
%
%
By using a mathematical software, the determinant of matrix $A(\lambda_1, \lambda_2)$ is given by
\begin{multline*}
\det A = 2 (-1)^{\frac{15}{8}} e^{-\frac{i (6 + 3 \lambda_1 + \lambda_2) \pi}{8} } \left( 1 + e^{i \lambda_1 \pi}\right) \left(-1 + e^{\frac{i \lambda_2 \pi}{2}}\right) \sec\left(\frac{ (1 + \lambda_1) \pi}{4}\right)\\
+4 (-1)^{\frac{3}{8}} e^{- \frac{i (\lambda_1 + \lambda_2) \pi}{8}} \left(-1 +    e^{\frac{i \lambda_1 \pi}{2}}\right) \left(1 - i e^{ \frac{i \lambda_2 \pi}{2}}\right).
\end{multline*}
 Note that the following graphics, with real and imaginary parts, of the determinant function $A(\lambda_1, \lambda_2)$,  helps us to see when the matrix $A$ is invertible.
 
\begin{figure}[h]
\centering
\begin{minipage}[c]{\textwidth}
	\includegraphics[width=8cm]{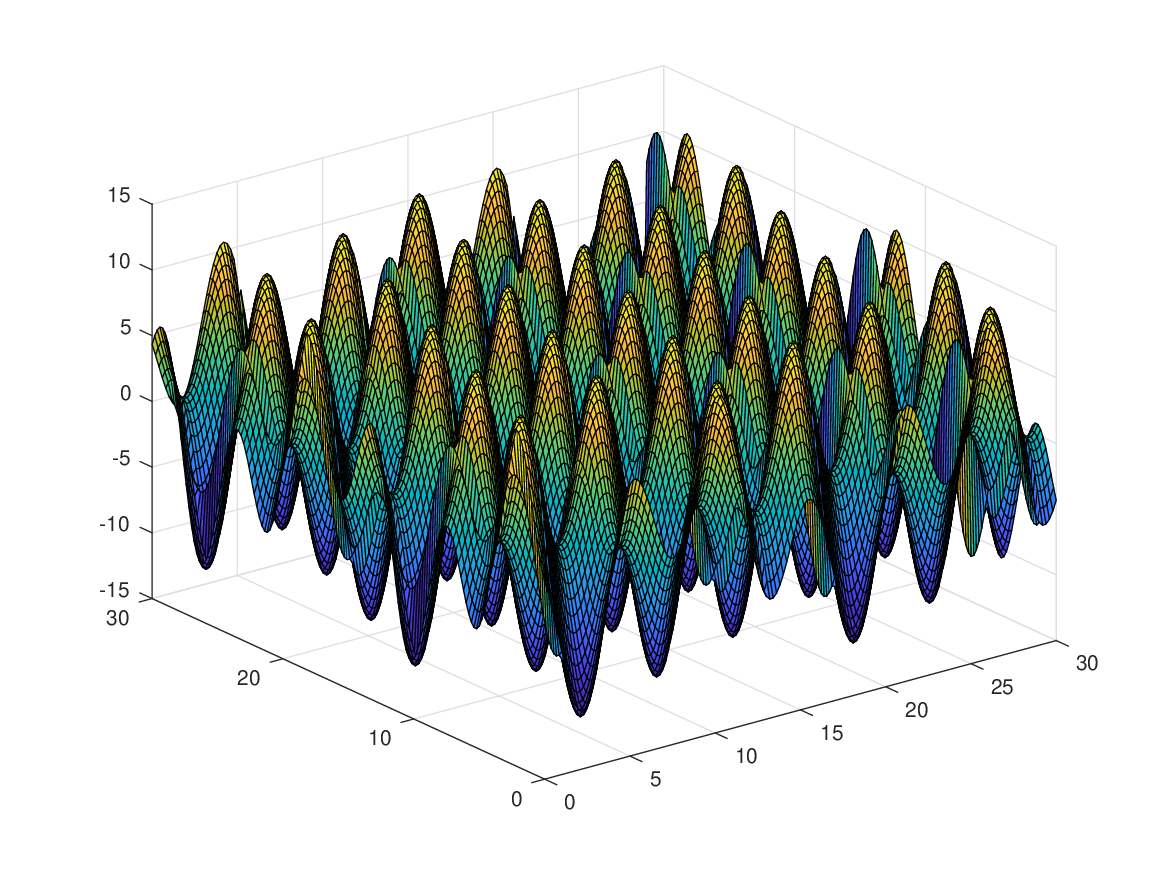}
	\includegraphics[width=8cm]{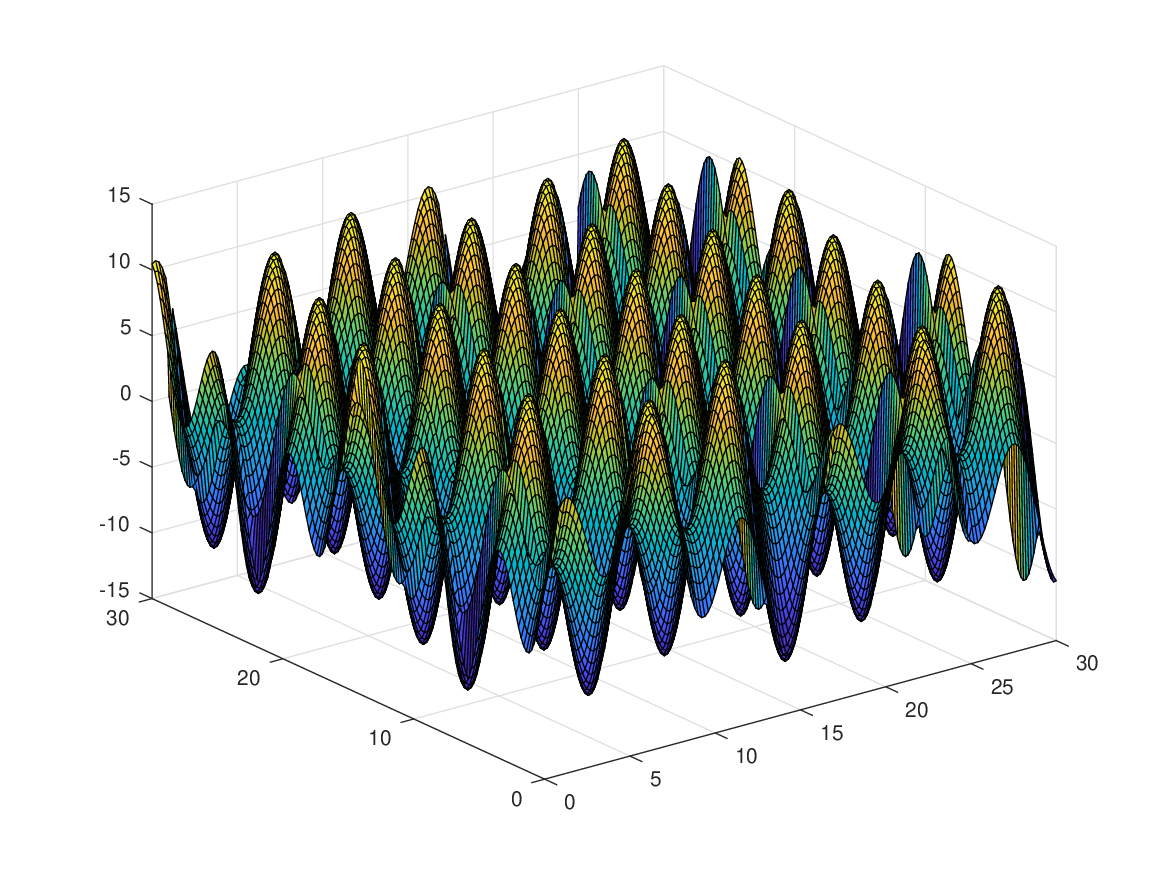}
\end{minipage}
	\caption{Real and imaginary part of $\det A$}\label{fig}
\end{figure}

Thus, matrix $A(\lambda_1,\lambda_2)$ is invertible if we get
\begin{multline}\label{det1}
\lambda_2\neq \frac{2}{\pi}\left( 2\pi n - i\log \left\lbrace\frac{-2(-1)^{\frac14}e{\frac{i\pi\lambda_1}{4}}+2(-1)^{\frac14}e^{\frac{3i\pi\lambda_1}{4}}+\left( e^{i\pi\lambda_1} +1\right)\sec\left(\frac{(1+\lambda_1)\pi}{4}\right) }{-2(-1)^{\frac34}e{\frac{i\pi\lambda_1}{4}}+2(-1)^{\frac34}e^{\frac{3i\pi\lambda_1}{4}}+\left( e^{i\pi\lambda_1} +1\right)\sec\left(\frac{(1+\lambda_1)\pi}{4}\right) } \right\rbrace \right)
\end{multline}
and
\begin{align}\label{eq:lambda condition}
\lambda_j \neq 1-4n, \quad  \lambda_j \neq 2-4n, \quad j=1,2,
\end{align}
for all $n \in \mathbb{Z}$. 

Figure \ref{fig} help us to see that there are infinity set of parameters which one satisfies the relations \eqref{det1} and \eqref{eq:lambda condition}.  In fact, for example, pick $\lambda_1\approx 0$ and $\lambda_2 \approx 1/3$. Thus, for $0\leq s<\frac12$, the choice of parameters $\lambda_1$ and $\lambda_2$ satisfying the following conditions
\begin{equation*}
-3<\lambda_j<\frac12,\ s+4b-2<\lambda_j<s+\frac12, \quad j=1,2,
\end{equation*}
ensures that Lemma \ref{edbf} holds. Thus, for a fix $s\in[0,\frac12)$, we can choose $\lambda_1$ and $\lambda_2$ as before and define the forcing functions $\gamma_1(t)$ and $\gamma_2(t)$ for any $\lambda_j$, $j=1,2$, given by
\begin{equation}\label{lambda} \left[\begin{array}{c}
\gamma_1(t)\\
\gamma_2(t)   \end{array} \right]=A^{-1}\left[\begin{array}{c}
f(t)-F(t,0) \\
\mathcal{I}_{\frac14}g(t)-  \mathcal{I}_{\frac14} \partial_xF(t,0)  \end{array} \right],
\end{equation}
which shows that formula \eqref{eq:solution} restricted on the set $(0,+\infty)\times (0,+\infty)$ satisfies $$(i\pt - \px^4)u = \lambda |u|^2u,$$
in the sense of distributions.

Thus, we define the solution operator by
\begin{equation}\label{eq:solution op}
\Lambda u(t,x)= \psi(t)\mathcal{L}^{\lambda_1}\gamma_1(t,x)+\psi(t)\mathcal{L}^{\lambda_2}\gamma_2(t,x)+\psi(t)F(t,x),
\end{equation}
where 
\[\left[\begin{array}{c}
\gamma_1(t)\\
\gamma_2(t)   \end{array} \right]=A^{-1}\left[\begin{array}{c}
f(t)-F(t,0) \\
\mathcal{I}_{\frac14}g(t)-  \mathcal{I}_{\frac14} \partial_xF(t,0)  \end{array} \right],\]
 $F(t,x)=e^{it\partial_x^4}\tilde{u}_0+\lambda\mathcal{D}(\psi_T|u|^2u)$ and $\psi$ is defined by \eqref{eq:cutoff}.

 
Recall the solution space $Z^{s,b}$, defined in Subsection \ref{sec:sol space}, under the norm
\[\norm{v}_{Z^{s,b}} = \sup_{t \in \R} \norm{v(t,\cdot)}_{H^s} + \sum_{j=0}^{1}\sup_{x \in \R} \norm{\px^jv(\cdot,x)}_{H^{\frac{2s+3-2j}{8}}} + \norm{v}_{X^{s,b} }.\]
The estimates obtained in Section \ref{sec:pre} together with estimates of Section \ref{sec:energy} and \eqref{eq:bilinear1} yield that
\[\|\Lambda u\|_{Z^{s,b}}\leq c(\|u_0\|_{H^s(\R^+)}+\|f\|_{H^{\frac{2s+3}{8}}(\R^+)}+\|g\|_{H^{\frac{2s+1}{8}(\R^+)}} ) + C_1{T^{\epsilon}}\norm{u}_{Z^{s,b}}^3,\]
for $\epsilon$ adequately small.
Similarly,
\[\|\Lambda u_1 - \Lambda u_2\|_{Z^{s,b}}\leq C_2{T^{\epsilon}}(\norm{u_1}_{Z^{s,b}}^2+\norm{u_2}_{Z^{s,b}}^2)\norm{u_1-u_2}_{Z^{s,b}},\]
for $u_1(0,x) = u_2(0,x)$.

Consider in $Z$ the ball defined by 
$
B=\{ u\in Z^{s,b}; \|u\|_{Z^{s,b}}\leq M\},
$
where $$M=2 c\left(\|u_0\|_{H^s(\R^+)}+\|f\|_{H^{\frac{2s+3}{8}}(\R^+)}+\|g\|_{H^{\frac{2s+1}{8}(\R^+)}}\right).$$ Lastly, choosing $T=T(M)$ sufficiently small, such that 
 \[\|\Lambda u\|_{Z^{s,b}}\leq M\]
 and
 \[ \|\Lambda u_1 - \Lambda u_2\|_{Z^{s,b}}\leq \frac12\norm{u_1-u_2}_{Z^{s,b}}, \] it follows that $\Lambda$ is a contraction map on $B$, and the proof of Theorem \ref{theorem1} is achieved. \qed
 
\begin{remarks}In what concerns our main result, Theorem \ref{theorem1}, the following remarks are now in order.
\begin{itemize}
	\item[i.]
	An important point to treat in dispersive system is the analysis of the scaling. For our case, that is,  for the biharmonic Schr\"odinger equation on the half-line, we have the following: If $u(t,x)$ is solution for IBVP \eqref{fourth} on $[0,T)\times(0,\infty)$, then,  for $\lambda>0$, the function $u_{\lambda}(t,x)=\lambda^2u(\lambda^4t,\lambda x)$ is solution for \eqref{fourth} on $[0,T/\lambda^4)\times(0,\infty)$ with initial-boundary conditions $u_{\lambda}(0,x)=\lambda^2u_0(\lambda x):=u_{0\lambda}$, $u_{\lambda}(t,0)=\lambda^2f(\lambda^4t):=f_{\lambda}$ and $u_{x,\lambda}(t,0)=\lambda^3g(\lambda^4t):=g_{\lambda}$.  A straightforward calculation gives
	\begin{multline}\label{scaling}
	\|u_{0\lambda}\|_{H^s(\mathbb{R}^+)}+\|f_{\lambda}\|_{H^{\frac{2s+3}{8}}(\mathbb{R}^+)}+\|g_{\lambda}\|_{H^{\frac{2s+1}{8}}(\mathbb{R}^+)}\\
	\lesssim\lambda^{3/2}\langle\lambda\rangle^s\|u_{0}\|_{H^s(\mathbb{R}^+)}+\langle\lambda\rangle^{\frac{2s+3}{2}}\|f\|_{H^{\frac{2s+3}{8}}(\mathbb{R}^+)}+\lambda\langle\lambda\rangle^{\frac{2s+1}{8}}\|g\|_{H^{\frac{2s+1}{8}}(\mathbb{R}^+)}.
	\end{multline}
	\item[ii.] In order to make the norms of our initial data $u_0$, $f$ and $g$ small, we re-scale the data $u_0$ and $g$ by choosing $\lambda$ adequately small, by using \eqref{scaling}. However, we can not re-scale the function $f$ since in \eqref{scaling} does not appear a positive power of $\lambda$.  To overcome this difficulty in our context,  we introduce the cut-off function $\psi_T$, defined by \eqref{eq:cutoff}, in the operator $\Lambda$ (see equation \eqref{eq:solution op}) to prove that  $\Lambda$ is thus a contraction, proving the main result of the article.
		\item[iii.] It is important to note that the scaling argument was successful  in the cases of the quadratic NLS equation \cite{Cavalcante}, KdV equation \cite{Holmerkdv} and Kawahara equation \cite{CaK} posed on the half-line.
	\item[iv.] Finally, in view of \eqref{eq:f}, \eqref{eq:g} and \eqref{lambda},  it is necessary to check $\gamma_i(t)$, $i=1,2$ to be well-defined in $H_0^{\frac{2s+3}{8}}(\R^+)$. However, it follows from Lemmas \ref{grupo}, \ref{duhamel} and \ref{edbf}, Propositions \ref{prop:bi1} and Lemmas \ref{sobolevh0} and \ref{alta}.
	\end{itemize}
		\end{remarks}

\section{Further comments and open problems}\label{further}
In this section, our plan is to present four problems that can be treated with the approach used in this article.

\subsection{Biharmonic NLS on star graphs} 
In a recent work \cite{CaCaGa}, the authors consider the biharmonic Schr\"odinger equation on star graphs, given by $N$ edges $(0,\infty)$ connected with a common vertex $(0,0, \cdots, 0)$ (see Figure \ref{fig2}), namely 
\begin{equation}\label{graph}
\begin{cases}
i\pt u_j - \px^4 u_j + \lambda |u_j|^2u_j=0, & (t,x)\in (0,T)  \times(0,\infty),\ j=1,2, ..., N\\
u_j(0,x)=u_{j0}(x),                                   & x\in(0,\infty),
\end{cases}
\end{equation}
with initial conditions $(u_1(0,x),u_2(0,x), ...,u_N(0,x))\in H^{s}(\R^+)$.

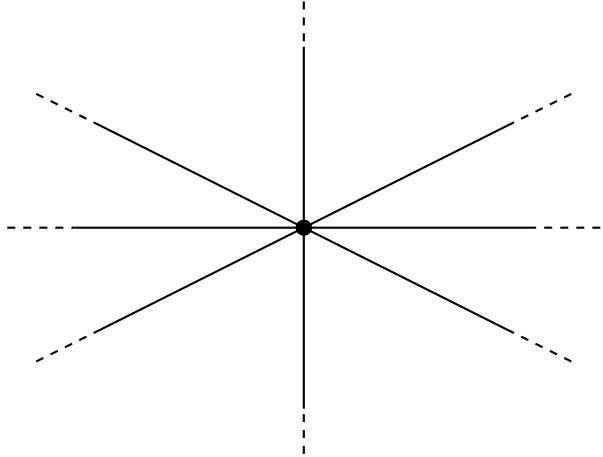
\begin{figure}[htp]
	\centering 
	\begin{tikzpicture}[scale=3]
	\draw[thick](-1,0)--(0,0);
	\draw [thick, dashed] (-1.3,0)--(-1,0);

	
		\draw [thick, dashed] (1.3,0)--(1,0);
	\draw[thick](1,0)--(0,0);
	
		\draw[thick](0,-0.8)--(0,0);
	\draw [thick, dashed] (0,-1)--(0,-0.8);
	
	
		\draw[thick](0,0.8)--(0,0);
	\draw [thick, dashed] (0,1)--(0,0.8);

	
	\draw[thick](0,0)--(-0.89,0.45);
	\draw [thick, dashed] (-0.89,0.45)--(-1.19,0.6);
	
	\draw[thick](0,0)--(-0.89,-0.45);
	\draw [thick, dashed] (-0.89,-0.45)--(-1.19,-0.6);
	\fill (0,0)  circle[radius=1pt];

	\draw[thick](0,0)--(0.89,0.45);
	\draw [thick, dashed] (0.89,0.45)--(1.19,0.6);
	\draw[thick](0,0)--(0.89,-0.45);
	\draw [thick, dashed] (0.89,-0.45)--(1.19,-0.6);
	\fill (0,0)  circle[radius=1pt];
	\end{tikzpicture}
	\caption{Star graphs connected with a common vertex $(0,0, \cdots, 0)$ }
	\label{fig2}
\end{figure}

For a better understanding,  we are interested in solving \eqref{graph} with the following three classes boundary conditions:
\begin{equation}\label{boundary1}
\begin{cases}
\partial_x^k u_1(t,0)=\partial_x^k u_2(t,0)=\cdots=u_N(t,0),\ k=0,1 & t\in(0,T)\\
\sum_{j=1}^n\partial_x^ku_j(t,0)=0,\ k=2,3\ & t\in(0,T);
\end{cases}
\end{equation}
\begin{equation}\label{boundary2}
\begin{cases}
\partial_x^k u_1(t,0)=\partial_x^k u_2(t,0)=\cdots=u_N(t,0),\ k=2,3 & t\in(0,T)\\
\sum_{j=1}^n\partial_x^ku_j(t,0)=0,\ k=0,1\ & t\in(0,T)
\end{cases}
\end{equation}
and
\begin{equation}\label{boundary3}
\begin{cases}
\partial_x^k u_1(t,0)=\partial_x^k u_2(t,0)=\cdots=u_N(t,0),\ k=0,3 & t\in(0,T)\\
\sum_{j=1}^n\partial_x^ku_j(t,0)=0,\ k=1,2\ & t\in(0,T).
\end{cases}
\end{equation}

The motivation of these boundary conditions and how we can choose it, follows the ideas contained in the paper of the second author \cite{Cav}, and are detailed in \cite{CaCaGa}. 

\subsection{Control theory} We split this section in two parts: Control theory for the biharmonic NLS on star graphs and on an unbounded domain, respectively.

\subsubsection{Control theory of Biharmonic NLS on star graphs} First, let us consider the controllability problem associated to \eqref{graph} with three possibilities of boundary conditions, namely, \eqref{boundary1}, \eqref{boundary2} and \eqref{boundary3}. Due of the recent development of graph theory for Korteweg-de Vries equation, in the following paragraph we present a few comments about this study.

In three interesting papers Ammari and Crepeau  \cite{AmCr},   Cavalcante \cite{Cav} and Mugnolo \textit{et al.} \cite{Noja} deal with the study of the KdV and Airy equations in graphs. In summary, in the first work, the authors proposed a model using the Korteweg--de Vries equation on a finite star-shaped network and proved the well-posedness of the system. Also, as main result of the work, by using properties of the energy, they showed that the solutions of the system decays exponentially to zero (as $t\to\infty$) and they studied an exact boundary controllability problem. In the second work, Cavalcante showed local well-posedness for the Cauchy problem associated with Korteweg--de Vries equation on a metric star graph. More precisely, he used the Duhamel boundary forcing operator, in the context of half-line, introduced by Colliander and Kenig \cite{CK} and Holmer \cite{Holmerkdv}  to achieve his result. Finally, Mugnolo \textit{et al.} obtained a characterization of all boundary conditions under which the Airy-type evolution equation $u_t=\alpha u_{xxx}+\beta u_x$, for $\alpha\in \R\setminus \{0\}$ and $\beta\in \R$ on star graphs,
generates contraction semigroups.

In this spirit, looking for the energy identity of the system \eqref{graph}, namely the $L^2$-energy, which one satisfies an equality given by
\begin{equation}\label{energy}
\begin{split}
E(u_1(T,x),u_2(T,x),\cdots, u_N(T,x))=&- \sum_{i=1}^N \int_0^T\text{Im}(\partial_x^3u_j(t,0)\overline{u}_j(t,0)) dt\\ &+\sum_{i=1}^N\int_0^T\text{Im}(\partial_x^2u_j(t,0)\partial_x\overline{u}_j(t,0)) dt\\
&-E(u_1(0,x),u_2(0,x),\cdots, u_N(0,x)),
\end{split}
\end{equation}
where $$E(u_1(t,x),u_2(t,x),\cdots, u_N(t,x)):=\sum_{i=1}^{N}\int_0^{+\infty}|u(t,x)|^2dx,$$
the following natural questions arise.

\vspace{0.1cm}
\noindent\textbf{Problem $\mathcal{A}$}: Which are the boundary conditions that we can impose in \eqref{boundary1}, \eqref{boundary2} and \eqref{boundary3} such that the energy is a non-increasing function of the time variable $t$?
\vspace{0.1cm}


\vspace{0.1cm}
\noindent\textbf{Problem $\mathcal{B}$}: If we can impose some boundary conditions such that the energy \eqref{energy} is a non-increasing function of the time variable $t$, is the system \eqref{graph}, with appropriated boundary conditions, asymptotically stable when the time tends to infinity?
\vspace{0.1cm}

Finally, the last problem in this subsection is the following one.

\vspace{0.1cm}
\noindent\textbf{Problem $\mathcal{C}$}: Can we find appropriated boundary controls such that the system \eqref{graph} is controllable in some sense?

\subsubsection{Control theory of Biharmonic NLS in unbounded domain} In the context of control in unbounded domain Faminskii \cite{Faminskii}, in a recent work, considered the initial-boundary value problem posed on infinite domain for Korteweg--de Vries equation. Precisely, he elected a function $f_0$ on the right-hand side of the equation as an unknown function, regarded as a control. Thus, the author proved that this function must be chosen such that the corresponding solution should satisfy certain additional integral condition.

Thus, this techniques, probably, work well for the following biharmonic NLS system
\begin{equation}\label{fourthd}
\begin{cases}
i\pt u +\gamma\px^4 u + \lambda |u|^2u=f_0(t)v(x,t), & (t,x)\in (0,T)  \times(0,\infty),\\
u(0,x)=u_0(x),                                   & x\in(0,\infty),\\
u(t,0)=h(t),\ u_x(t,0)=g(t)& t\in(0,T),
\end{cases}
\end{equation}
for $\gamma, \lambda\in\mathbb{R}$, where $v$ is a given function and $f_0$ is an unknown control function. Therefore, the issue here is:

\vspace{0.1cm}
\noindent\textbf{Problem $\mathcal{D}$}: Is the \eqref{fourthd} controllable in the sense of Faminskii's work?  Namely, can we find a pair $\{f_0, u\}$, satisfying an appropriated additional integral conditions  (for details see \cite{Faminskii})?

\subsection*{Acknowledgments} The authors wish to thank the referee for his/her valuable comments which improved this paper. R. A. Capistrano--Filho was supported by CNPq (Brazil) grants 306475/2017-0, 408181/2018-4, CAPES-PRINT (Brazil) grant 88881.311964/2018-01 and Qualis A - Propesq (UFPE). F. A. Gallego was partially supported under projects SIGP 58907 and 45511.  This work was carried out during some visits of the authors to the Federal University of Pernambuco and Universidad Nacional de Colombia - Sede Manizales. The authors would like to thank both Universities for its hospitality. 

\end{document}